\documentclass[10pt,a4paper]{article}
\date{}
\usepackage{graphicx}
\usepackage{epstopdf}
\usepackage{epsfig}
\usepackage{fullpage} 

\usepackage[T1]{fontenc} 
\usepackage[utf8]{inputenc}

\usepackage{amsmath} 

\usepackage{amsfonts}
\usepackage{amssymb}
\usepackage{pstricks}
\usepackage{mathabx}
\usepackage{enumitem}

\newtheorem{theorem}{Theorem}[section]
\newtheorem{lemma}[theorem]{Lemma}
\newtheorem{proposition}[theorem]{Proposition}

\newtheorem{remark}[theorem]{Remark}

\newenvironment{proof}[1][Proof]{\begin{trivlist}
\item[\hskip \labelsep {\bfseries #1}]}{\end{trivlist}}

\newcommand{\modd}[1]{\vert #1\vert^2}

\newcommand{\nablag}{\nabla^{\gamma}}

\newcommand{\Vu}{\overline{V}}
\newcommand{\Vd}{\underline{V}}

\newcommand{\Vk}{\overline{V}_k}
\newcommand{\zetk}{\zeta_{k}}

\newcommand{\ds}{\displaystyle}
\newcommand{\D}{\vert D^\gamma\vert}

\newcommand{\R}{\mathbb{R}}

\newcommand{\Tb}{\mathcal{T}_b}
\newcommand{\Ud}{U^\delta}

\newcommand{\Vde}{\overline{V}^\delta}
\newcommand{\zetade}{\zeta^\delta}
\newcommand{\zetadede}{\zeta^\delta}

\newcommand{\dt}{\partial_t}
\newcommand{\hbi}{\frac{1}{h_b}}
\newcommand{\nablago}{{\nabla}^{\gamma\perp}}
\def \epsilon {\varepsilon}
\newcommand{\Bi}{\mathcal{B}}
\newcommand{\A}{\mathcal{A}}
\newcommand{\Vda}{\Vde_{(\alpha)}}
\newcommand{\zetaa}{\zeta_{(\alpha)}}
\newcommand{\zetadea}{\zetade_{(\alpha)}}
\newcommand{\qa}{q_{(\alpha)}^\delta}
\newcommand{\Wzf}{\infty}
\newcommand{\Wuf}{{W^{1,\infty}}}
\newcommand{\Wdf}{{W^{2,\infty}}}

\begin{document}

\title{The Cauchy problem on large time for a Boussinesq-Peregrine equation with large topography variations}

\author{Mésognon-Gireau Benoît\footnote{UMR 8553 CNRS, Laboratoire de Mathématiques et Applications de l’Ecole Normale Supérieure, 75005 Paris, France. Email: benoit.mesognon-gireau@ens.f}}
\maketitle

\begin{abstract}
We prove in this paper a long time existence result for a modified Boussinesq-Peregrine equation in dimension $1$, describing the motion of Water Waves in shallow water, in the case of a non flat bottom. More precisely, the dimensionless equations depend strongly on three parameters $\epsilon,\mu,\beta$ measuring the amplitude of the waves, the shallowness and the amplitude of the bathymetric variations respectively. For the Boussinesq-Peregrine model, one has small amplitude variations ($\epsilon = O(\mu)$). We first give a local existence result for the original Boussinesq Peregrine equation as derived by Boussinesq (\cite{boussinesq1},\cite{boussinesq2}) and Peregrine (\cite{peregrine}) in all dimensions. We then introduce a new model which has formally the same precision as the Boussinesq-Peregrine equation, and give a local existence result in all dimensions. We finally prove  a local existence result on a time interval of size $\frac{1}{\epsilon}$ in dimension $1$ for this new equation, without any assumption on the smallness of the bathymetry $\beta$, which is an improvement of the long time existence result for the Boussinesq systems in the case of flat bottom ($\beta=0$) by \cite{saut}.
\end{abstract}

\section{Introduction}
We recall here the context of the Water Waves problem, and introduce all the meaningful dimensionless parameters of this problem. We then present the shallow water regime and more specifically the Boussinesq-Peregrine regime. We finally introduce the different results proved in this paper, which are local existence theorems for the Boussinesq-Peregrine equations on different time scales and with different conditions.

\subsection{The Water Waves problem}

The motion, for an incompressible, inviscid and irrotationnal fluid occupying a domain $\Omega_t$ delimited below by a fixed bottom and above by a free surface is commonly referred to as the Water Waves problem. It is described by the following quantities:\begin{itemize}[label=--,itemsep=0pt]
\item the velocity of the fluid $U=(V,w)$, where $V$ and $w$ are respectively the horizontal and vertical components ;
 \item the free top surface profile $\zeta$  ;
 \item the pressure $P.$ 
\end{itemize}
All these functions depend on the time and space variables $t$ and $(X,z) \in\Omega_t$, which is the domain occupied by the water. More precisely, there exists a function $b:\mathbb{R}^d\rightarrow \mathbb{R}$ such that the domain of the fluid at the time $t$ is given by  $$\Omega_t = \lbrace (X,z)\in\mathbb{R}^{d+1},-H_0+ b(X) < z <\zeta(t,X)\rbrace,$$ where $H_0$ is the typical depth of the water. The unknowns $(U,\zeta,P)$ are governed by the Euler equations: 
\begin{align}
\begin{cases}
\partial_t U +  U\cdot\nabla_{X,z} U = -\frac{1}{\rho} \nabla P-g e_z \text{ in } \Omega_t\\
\mbox{\rm div}(U) = 0 \text{ in } \Omega_t\\
\mbox{\rm curl}(U) = 0 \text{ in } \Omega_t .
\label{c3:euler}
\end{cases}\end{align}

These equations are completed by boundary conditions: 
\begin{align}
\begin{cases}
\partial_t \zeta +\underline{V}\cdot\nabla\zeta - \underline{w} = 0  \\
U\cdot n = 0 \text{ on } \lbrace z=-H_0+ b(X)\rbrace \\
P = P_{atm}\text{ on }  \lbrace z=\zeta(X)\rbrace. \label{c3:boundary_conditions}
\end{cases}
\end{align}
In these equations, $\underline{V}$ and $\underline{w}$ are the components of the velocity evaluated at the surface. The vector $n$ in the last equation stands for the upward normal  vector at the bottom $(X,z=-H_0+b(X))$, and $e_z$ is the unit upward vector in the vertical direction. We denote $P_{atm}$ the constant pressure of the atmosphere at the surface of the fluid, $\rho$ the (constant) density of the fluid, and $g$ the acceleration of gravity. The first equation of \eqref{c3:boundary_conditions} states the assumption that the fluid particles do not cross the surface, while the second equation of \eqref{c3:boundary_conditions} states the assumption that they do not cross the bottom. The equations \eqref{c3:euler} with boundary conditions \eqref{c3:boundary_conditions} are commonly referred to as the free surface Euler equations.

\subsection{The dimensionless parameters}

Since the properties of the solutions of the Water Waves problem depend strongly on the characteristics of the flow, it is more convenient to non-dimensionalize the equations by introducing some characteristic lengths of the wave motion: \begin{enumerate}[itemsep=0pt,label=(\arabic*)]
\item The characteristic water depth $H_0$.
\item The characteristic horizontal scale $L_x$ in the longitudinal direction.
\item The characteristic horizontal scale $L_y$ in the transverse direction (when $d=2$).
\item The size of the free surface amplitude $a_{surf}$.
\item The size of bottom topography $a_{bott}$.
\end{enumerate}

Let us then introduce the dimensionless variables: $$x'=\frac{x}{L_x},\quad y'=\frac{y}{L_y},\quad \zeta'=\frac{\zeta}{a_{surf}},\quad z'=\frac{z}{H_0},\quad b'=\frac{b}{a_{bott}},$$ and the dimensionless variables: $$t'=\frac{t}{t_0},\quad P'=\frac{P}{P_0},$$ where $$t_0 = \frac{L_x}{\sqrt{gH_0 }},\quad P_0 = \rho gH_0.$$ 

After rescaling, four dimensionless parameters appear in the Euler equation. They are 
\begin{align*}
\frac{a_{surf}}{H_0} = \epsilon, \quad \frac{H_0^2}{L_x^2} = \mu,\quad \frac{a_{bott}}{H_0} = \beta,\quad \frac{L_x}{L_y} = \gamma,
\end{align*}
where  $\epsilon,\mu,\beta,\gamma$ are commonly referred to  respectively as "nonlinearity", "shallowness", "topography" and "transversality" parameters. The free surface Euler equations \eqref{c3:euler} and \eqref{c3:boundary_conditions} become after rescaling (we omit the "primes" for the sake of clarity):
\begin{align}
\begin{cases}
\partial_t U + \epsilon(V \cdot\nablag+\frac{1}{\mu} w\partial_z)U = -\frac{1}{\epsilon}\nablag P,\\
\partial_t \zeta -\sqrt{1+\epsilon^2\modd{\nablag\zeta}} U\cdot n = 0  \label{c3:dimensionless_euler}
\end{cases}
\end{align}
where we used the following notations:
\begin{align*}
\nabla^{\gamma} = {}^t(\partial_x,\gamma\partial_y) \quad \text{ if } d=2 \text{ and } \nabla^{\gamma} = \partial_x \quad \text{ if } d=1, 
\end{align*} and we recall that the unknown is the velocity $U=(V,w)$ where $V$ and $w$ are respectively the horizontal and vertical components of the velocity. The equations \eqref{c3:dimensionless_euler} with boundary conditions for the pressure and the velocity are commonly referred to as the "dimensionless free surface Euler equations".

\subsection{The Shallow Water regime}\label{c3:shallow_section}

When the shallowness parameter $\mu = \frac{H_0^2}{L_x^2}$ is small, it is possible to use a simplified equation in order to study the Water Waves problem. More precisely, at first order with respect to $\mu$, the horizontal velocity $V$ becomes columnar, which means that \begin{equation}V = \Vu + \mu R\label{c3:columnar}\end{equation} where $\Vu$ stands for the vertical average of the horizontal velocity $$\Vu(t,X) = \frac{1}{h(t,X)}\int_{1-\beta b(X)}^{\epsilon\zeta(t,X)} V(t,X,z)dz$$ and $h$ is the height of the water $h(t,X) = 1+\epsilon\zeta(t,X)-\beta b(X)$. We do not give precise estimate for the residual $R$ of \eqref{c3:columnar} in Sobolev norm here. Lagrange \cite{lagrange}, and later Saint-Venant \cite{saintvenant1} derived from the Euler equations and under the assumption that the pressure is hydrostatic\footnote{ The pressure is hydrostatic if, in dimensional form, $P(X,z) = P_{atm}-\rho g(z-\zeta)$. This is always true at the leading order in $\mu$. See for instance \cite{david} Section 5.5} the following Shallow-Water equation expressed in term of unknowns $(\Vu,\zeta)$:
\begin{align}
\begin{cases}\label{c3:shallow_water}
\partial_t\zeta+\nablag\cdot(h\Vu) = 0 \\
\partial_t V+\nablag\zeta+\epsilon(\Vu\cdot\nablag)\Vu = 0,
\end{cases}
\end{align} with initial data $(\Vu,\zeta)_{\vert t=0} = (\Vu^0,\zeta^0)$. The Shallow-Water equations \eqref{c3:shallow_water} are a typical example of quasilinear symmetrizable system (the symmetrization is done by multiplying the second equation by $h$). The local existence result for such a system is classical, and is done for example in \cite{taylor3} Chapter XVI (see also \cite{benzoniserre}). The Shallow-Water equation is said to be consistent at the first order in $\mu$ with the Water-Waves equations, which means that formally, one has $$\text{ water-waves equation = shallow-water equation} + O(\mu).$$ It formally means that if one is interested in working with a shallow water (i.e. with small values of $\mu$), one can get rid of all the terms of size $\mu$ in the Water-Waves equations and obtain the simplified model of the Shallow-Water equation. Alvarez-Sameniago and Lannes  \cite{alvarez}, and Iguchi \cite{iguchi2009} fully justified the Shallow-Water model by proving the mathematical convergence of the Water-Waves equation to the Shallow-Water equation. More precisely, one has the following result:
\begin{theorem}\label{c3:consistance_def}There exists $N\geq 1$, such that for all $(\Vu^0,\zeta^0)\in H^{N}(\R^d)^{d+1}$, there exists $T>0$ such that: \begin{enumerate}[itemsep=0pt,label=(\arabic*)] \item There exists a unique solution $(\zeta^E,U^E)$ defined on $[0;T[$ to the free surface Euler equation \eqref{c3:euler}, \eqref{c3:boundary_conditions} such that $(\zeta^E,\Vu^E)\in H^N(\R^d)^{d+1}$, and $(\zeta^E,\Vu^E)(0) = (\zeta^0,\Vu^0)$. \item There exists a unique solution $(\zeta^{SW},\Vu^{SW})\in C([0;T[;H^N(\R^d)^{d+1})$ to the Shallow-Water equation \eqref{c3:shallow_water} with initial conditions $(\zeta^0,\Vu^0)$.
\item One has, for all $t\in[0;T[$:
 $$\vert (\zeta^E,\Vu^E)(t)-(\zeta^{SW},\Vu^{SW})(t)\vert_{H^N} \leq C(\vert(\zeta^0,\Vu^0)\vert_{H^N}) \mu t.$$
\end{enumerate}

\end{theorem}
This theorem implies that the error made by using the solutions of the simplified model of Shallow-Water instead of the solutions of the free-surface Euler equations is "of size $\mu$".\par\vspace{\baselineskip}

At the second order with respect to $\mu$, one can derive several models for the Water-Waves problem which are formally more precise than the Shallow-Water equation. We consider in this paper the case where $$\epsilon = O(\mu),$$ which corresponds to a small amplitude model in the Shallow-Water regime. Boussinesq (\cite{boussinesq1},\cite{boussinesq2}) derived the following model for flat bottoms, and later Peregrine \cite{peregrine} for nonflat bottoms:
\begin{align}
\begin{cases}
\partial_t\zeta+\nablag\cdot(h\Vu)=0 \\
[1+\mu\Tb]\partial_t\Vu+\epsilon(\Vu\cdot\nablag)\Vu+\nablag\zeta = 0 \label{c3:boussinesq_equation1}
\end{cases}
\end{align}
where $\Tb$ is the following operator : 
$$\forall\Vu,\qquad \Tb\Vu = -\frac{1}{3h_b}\nablag(h_b^3\nablag\cdot\Vu)+\frac{\beta}{2h_b}[\nablag(h_b^2\nablag b\cdot\Vu)-h_b^2\nablag b\nablag\cdot\Vu]+\beta^2\nablag b\nablag b\cdot\Vu,$$ with the notation  $$h_b = 1-\beta b.$$ See also \cite{david} for a complete proof of the formal derivation of this equation. This equation is known as the Boussinesq-Peregrine equation, and has been used a lot in applications to coastal flows. In the case of a flat bottom, one has $\beta = 0$ and $\Tb = -\frac{1}{3}\nablag\nablag\cdot$. The equation \eqref{c3:boussinesq_equation1} can then be seen as a particular case of a Boussinesq system (see \cite{bonachen} for the $1d$ case, and  \cite{bonalannes} for the $2d$ case). \par\vspace{\baselineskip}

The Boussinesq-Peregrine model is a good compromise for numerical simulation between the precision of the approximation of the Water-Waves problem, and the simplicity of the equations. Indeed, the Boussinesq-Peregrine equation preserves the dispersive nature of the Water-Waves equation. To understand this statement, one can look for plane wave solutions $(\zeta,\Vu) = (\zeta^0,\Vu^0)e^{i(k\cdot X-\omega(k) t)}$ of the linearized Boussinesq-Peregrine equation \eqref{c3:boussinesq_equation1} and finds solutions with a dispersive relation: $$\omega(k)^2 = \frac{\vert k^\gamma\vert^2}{1+\frac{\mu}{3}\vert k^\gamma\vert^2},$$ with $k^\gamma = {}^t(k_1,\gamma k_2)$. Thus the group velocity $c = \frac{\omega(k)}{\vert k^\gamma\vert}$ of the water waves depends on the frequency, which is a definition of dispersion. The Shallow-Water equation, however, is not a dispersive equation, at least in dimension $d=1$, since one would find a group velocity of the water waves equals to $1$. One could also derive an even more precise model at the $O(\mu^2)$ order than the Boussinesq-Peregrine model, without any assumption on the smallness of $\epsilon$, which is called the Green-Naghdi equation (see \cite{serre} for the case $d=1$, \cite{sugardner}, \cite{seabra} for the case of non-flat bottom and also under the name of "fully nonlinear Boussinesq equations" in \cite{wei}; see also \cite{miles}). It has a very similar form as one of the Boussinesq-Peregrine equation, but with $h_b$ replaced by $h$ in the definition of the operator $\Tb$. Therefore, though the Green-Naghdi model should be more precise than the Boussinesq-Peregrine equation, its numerical implementation leads to the computation at each time step of the inverse of $I+\mu\Tb$ (which now depends on $h$, and therefore on the time), which increases the computational complexity\footnote{Note however that a variant of the Green-Naghdi equation has where the operator to invert in time is time independent has been recently derived recently in \cite{lannesmarche}}. 

\subsection{Long time existence for the Water-Waves models}
We are interested in the dependence of the existence time for the solutions of systems like \eqref{c3:shallow_water} and \eqref{c3:boussinesq_equation1} with respect to the parameter $\epsilon$. For such quasilinear equations with an $\epsilon$ factor on the nonlinearity, the "good" time existence should be of size $\frac{1}{\epsilon}$. Let us explain this statement by studying the one dimensional Burgers equation:
\begin{align*}
\begin{cases}
\partial_t u+\epsilon u\partial_x u = 0\\
u(0,x) = u_0(x)
\end{cases}
\end{align*}
where we assume some reasonable regularity on $u_0$. Using the method of characteristics to solve it, we find that characteristics are of the form $$x(t) = \epsilon u_0(x_1)t+x_1$$ with $x_1\in\R$. The solutions do not exist globally in the general case because the characteristics intersect themselves. Let us estimate the time at which they intersect. Let us consider two characteristics $x(t) = \epsilon u_0(x_1)t+x_1$ and $\epsilon u_0(x_2)t+x_2$. They intersect at time $$t=-\frac{x_2-x_1}{\epsilon(u_0(x_2)-u_0(x_1))}.$$ Therefore, the existence time for the solutions is $$T = \underset{x_1,x_2\in\R}{\min} - \frac{x_2-x_1}{\epsilon(u_0(x_2)-u_0(x_1))} = - \frac{1}{\epsilon\underset{x\in\R}{\min} u_0'(x)}  $$ and therefore is of size $\frac{1}{\epsilon}$.\par\vspace{\baselineskip}

The standard theory for quasilinear symmetrizable systems gives the local existence of solutions $(\Vu,\zeta)$ of the Shallow-Water equations \eqref{c3:shallow_water} on the space $C([0;T];H^N(\R^d)^{d+1})$, with $N$ large enough, and gives an explosion criterion: one has $T<\infty$ if and only if $\underset{t\rightarrow T}{\lim} \vert (\Vu,\zeta)(t)\vert_{W^{1,\infty}} =+\infty$. If one could prove an estimate of the form \begin{equation}\vert (\Vu,\zeta)(t)\vert_{H^N} \leq g(\vert (\Vu,\zeta)\vert_{H^N}) t\epsilon\label{c3:estimate_attendue} \end{equation} for a smooth non decreasing function $g$, one would get by a continuity argument that the solutions exist on an interval of size $\frac{1}{\epsilon}$. An $L^2$ estimate of the form \eqref{c3:estimate_attendue} is easy to obtain, because multiplying the second equation of \eqref{c3:shallow_water} symmetrizes both first space derivatives order terms and quantities of size $\epsilon$. Indeed, one can differentiate with respect to time the "energy" $$E(\zeta,\Vu) = \frac{1}{2}\vert\zeta\vert_2^2+\frac{1}{2}( h\Vu,\Vu)_2$$ to get $$\frac{dE}{dt} = (\partial_t\zeta,\zeta)_2+(\partial_t\Vu,h\Vu)_2+\frac{1}{2}((\partial_t h)\Vu,\Vu)2$$ and replace $\partial_t(\zeta,\Vu)$ by their expressions given in the equation \eqref{c3:shallow_water}. The terms of order $1$ are transparent in the energy estimate because they cancel one another, thanks to the "symmetry" of the system, and one gets $$\frac{dE}{dt}=0$$ which is even better than needed. However, if one differentiates the equation \eqref{c3:shallow_water} with respect to space variable, one finds a system of the form ($\partial$ denotes here any space derivative of order one):
\begin{align*}
\begin{cases}\partial_t \partial\Vu + \nablag\cdot((\partial h) \Vu)+\nablag\cdot(h\partial(\Vu)) =0 \\
\partial_t \partial\zeta + \epsilon(\partial\Vu)\cdot\nablag\Vu + \epsilon \Vu\cdot\nablag \partial \Vu +\nablag\partial \zeta =0.
\end{cases}
\end{align*}
It is not possible anymore to make this system symmetric in order to cancel the terms that are not of size $\epsilon$ in the energy estimates.  Indeed, the term $\partial h$ is not of size $\epsilon$, since $h=1+\epsilon\zeta-\beta b$, and thus also depends on $\beta$. The $H^N$ norms of the unknowns are then not easily controlled by terms of size $\epsilon$, which prevent us from proving directly an estimate of the form \eqref{c3:estimate_attendue}. The long existence for this system is therefore tied to a singular perturbation problem with variable coefficients. \par\vspace{\baselineskip}

Long time existence results for  similar types of equations have been proved for example by Schochet in \cite{schochet} for the compressible Euler equation in a bounded domain with well-prepared data, or by Schochet-Métivier in \cite{schochetmetivier} for the Non-isentropic Euler equation with general data. Alazard  (\cite{alazard2}) proved a long time existence result for the non-isentropic compressible Euler equation, in the case of bounded and unbounded domains. Bresch-Métivier (\cite{bresch_metivier}) proved that for $N>d/2+1$ and $(\zeta^0,\Vu^0)\in H^N(\R^d)^{d+1}$, there exists a unique solution $(\zeta,\Vu) \in C([0;\frac{T}{\epsilon}];H^N(\R^d)^{d+1})$ to the equation \eqref{c3:shallow_water} with initial data $(\zeta^0,\Vu^0)$ where $T$ only depends on the norm of the initial data, even if $\beta$ is not assumed to be small. \par\vspace{\baselineskip}

In the case of a flat bottom, as explained in Section \ref{c3:shallow_section}, the Boussinesq-Peregrine equation can be seen as a particular case of the Boussinesq systems. Saut and Li (\cite{saut}) proved the local existence on a large time interval of size $\frac{1}{\epsilon} $ for most of these systems, when the bottom is flat. There is, in our knowledge, no local existence results in the literature in the case of non flat bottoms. A local existence result for the Boussinesq-Peregrine equation \eqref{c3:boussinesq_equation1} on a time $\frac{T}{\epsilon}$ with $T$ independent on $\mu$ would be important to fully justify this model, and get a convergence result similar to Theorem \ref{c3:consistance_def}.

\subsection{Main result}\label{c3:intro_explication}In \cite{benoit}, a large time existence result is proved for the Water-Waves equations in presence of large topography, extending the result of \cite{alvarez} which holds for small topography variations\footnote{However this result needs the presence of a small surface tension in the model} ($\beta=O(\epsilon))$). Coupling this result with the one of \cite{bresch_metivier}, one can prove that the Shallow-Water equations are consistent at order $1$ in $\mu$ with the full Water-Waves equations, on a time interval of size $\frac{1}{\epsilon}$, even in the case of large topography variations ($\beta = O(1)$). A similar result between the Water-Waves equations and the Boussinesq-Peregrine equation \eqref{c3:boussinesq_equation1} would be a new step in the mathematical justification of the Water-Waves models. \par\vspace{\baselineskip} We first prove in this paper a well posedness result for the equation \eqref{c3:boussinesq_equation1} on a time $O(1)$ (Theorem \ref{c3:boussi_theorem} below). Though not optimal as we shall see, such a local well posedness result did not seem to be available in the literature for non flat bottoms. The proof consists in a particular adaptation of the proof of local existence for symmetrizable quasilinear systems. The difficulty is that this system is not easily symmetrizable. In the case of the Shallow-Water equations \eqref{c3:shallow_water}, as explained in Section \ref{c3:shallow_section}, one should multiply the second equation by $h$ to get a "symmetric" system. However, in the case of the Boussinesq-Peregrine equation \eqref{c3:boussinesq_equation}, multiplying the second equation by $h$ indeed symmetrizes the system with respect to order one space derivatives. But the operator $h(I+\mu\Tb)$ is not symmetric. Therefore, $(h(I+\mu\Tb)\partial_t\Vu,\Vu)_2$ is not equal to $\frac{d}{dt} (h(I+\mu\Tb)\Vu,\Vu)_2$. It implies the presence of some commutators terms between $h\Tb$ and $\partial_t$ in the energy estimates which are difficult to control. However, the operator $h_b(I+\mu\Tb)$ is symmetric and we have the equivalence $$(h_b(I+\mu\Tb)\Vu,\Vu)_2 \sim \vert\Vu\vert_2^2+\mu\vert\nablag\cdot\Vu\vert_2^2.$$ Therefore, a "good" energy for the Boussinesq-Peregrine equation seems to be \begin{equation}E(\zeta,\Vu) = \frac{1}{2}\vert\zeta\vert_2^2+\frac{1}{2}(h_b(I+\mu\Tb)\Vu,\Vu)_2.\label{c3:energie_boussinesq}\end{equation} But multiplying the second equation by $h_b$ does not properly symmetrize the system with respect to space derivatives. More precisely, it symmetrizes them up to an $\epsilon\zeta\nablag \cdot \Vu$ factor. In the Boussinesq-Peregrine regime, one has $\epsilon = O(\mu)$ and therefore this term is actually controlled by the energy... It yields to a local existence result for a time interval of size $O(1)$ and not $O(\frac{1}{\epsilon})$. This is the purpose of Theorem \ref{c3:boussi_theorem}. \par\vspace{\baselineskip}

One then looks for an improved time of existence for the Boussinesq-Peregrine equation. In order to do so, one could try to use an adaptation of the proof of the long time existence result by Bresch-Métivier \cite{bresch_metivier}.  The idea of this proof is to have energy estimates of the form $$E(t)\leq C(E)t\epsilon+C_0$$ for some constant $C(E)$ which depends on the energy, and $C_0$ which only depends on initial data, where $E$ is an energy. One can then conclude by a continuity argument that the energy stays bounded on an interval of size $\frac{1}{\epsilon}$. Let us detail this idea on a simplified model of equation of the form:
\begin{equation}
\partial_t u+\epsilon (u\partial_x) u + L(\epsilon u,a(x))u = 0
\label{c3:equation_model}
\end{equation}
where $L(\epsilon u,a(x)) : \R^d \longrightarrow \R^d$ is a linear, elliptic, antisymmetric operator of order $1$. One computes: \begin{align*}\frac{d}{dt}\frac{1}{2}\vert u(t)\vert_2^2&= (\partial_t u,u)_2\\
&= (Lu,u)_2+\epsilon((u\partial_x) u,u)_2\end{align*}
using the equation to replace $\partial_t u$ by its expression. Since the operator $L + \partial_x$ is antisymmetric, $\frac{d}{dt}\frac{1}{2}\vert u(t)\vert_2^2$ vanishes, which is even better than being of size $\epsilon$. However, it does not stand true for higher order estimates. Looking for an estimate on $\frac{d}{dt}\vert u(t)\vert_{H^N}$ for some $N>0$,  one differentiates the equation \eqref{c3:equation_model} and finds a system of the form:
$$\partial_t \partial_x u +\epsilon (u \partial_x)\partial_x u +\epsilon [u\partial_x,\partial_x] u + L(\epsilon u, a(x)) \partial_x u + \epsilon (\partial_x u) dL_1(\epsilon u,a(x))u + (\partial_x a)dL_2(\epsilon u,a(x))u=0,$$
where we denoted $dL_i(\epsilon u,a(x))$ the differential of $L$ with respect to the $i-th$ variable, at the point $(\epsilon u,a(x))$, for $i=1,2$. Due to an extra term $(\partial_x a)dL_2(\epsilon u,a(x))u$, the energy estimates involve terms which are not of size $\epsilon$. This problem does not appear for time derivatives: if one differentiates the equation $\eqref{c3:equation_model}$ with respect to time, one finds:
$$\partial_t (\partial_t u)+\epsilon u\partial_x (\partial_t u) +\epsilon[u\partial_x,\partial_t] u+L(\epsilon u, a(x))\partial_t u+\epsilon(\partial_t u )dL_1(\epsilon u,a(x))u.$$ Therefore, one can find an energy estimate of the form $$\vert (\partial_t^k u)(t)\vert_2 \leq \vert (\partial_t^k u)(0)\vert_2  + \epsilon tC(\vert(\partial_t^k u)(t)\vert_2),$$ for all $k\leq N$. In order to find a similar energy estimate in $H^N$, one uses the equation, which gives an expression of space derivatives with respect to time derivatives:
$$L(\epsilon u,a(x))  u = -\partial_t  u  -\epsilon  u\partial_x u$$ and using the previous estimate for time derivatives, and the ellipticity of $L$, one gets:
$$\vert u(t)\vert_{H^1} \leq C(\vert u(t)\vert_{H^1})\epsilon t +C_0$$ for some constant $C_0$ and a non decreasing smooth function $C$. One can do the same for higher order Sobolev estimates, by considering higher order time derivatives $\partial_t^k$ and using a finite induction on $k$. By a continuity argument, an estimate of the form $$\vert u(t)\vert_{H^N} \leq C(\vert u(t)\vert_{H^1})\epsilon t+C_0$$ implies that the $H^N$ norm of $u$ stays bounded on an interval of size $\frac{1}{\epsilon}$. \par\vspace{\baselineskip}

This technique only works if time and space derivatives have the same "order". More precisely, for the Shallow-Water equation \eqref{c3:shallow_water}, the time derivatives are equal to sum of terms involving one space derivative. This is not the case for the Boussinesq-Peregrine equation \eqref{c3:boussinesq_equation1}. Indeed, in the second equation, $(I+\mu\Tb)\partial_t\Vu$ is equal to one space derivative order terms, while $(I+\mu\Tb)$ is of order two. It leads to issues if one tries to use the equation to control space derivatives by time derivatives and tries to recover an estimate of size $\epsilon$ for the space derivatives. For example, the second equation of \eqref{c3:boussinesq_equation1} provides $$\nablag\zeta = -(I+\mu\Tb)\partial_t \Vu+\epsilon\times\text{ other terms },$$ and $\Tb$ is an order two operator with respect to space. It is therefore not clear that $\Tb \partial_t\Vu$ is controlled by the energy \eqref{c3:energie_boussinesq}. \par \vspace{\baselineskip}

To overcome this problem, we introduce a modified equation, which is consistent with the Boussinesq-Peregrine equation \eqref{c3:boussinesq_equation1} at the $O(\mu^2)$ order (and therefore with the Water-Waves equations). Such equation would have a proper structure adapted to the use of the method used by Bresch-Métivier in \cite{bresch_metivier}. The approach of modifying the equation without changing the consistency, in order to improve the structure of the equation has been used for example by Israwi in \cite{israwi} for the Green-Naghdi equation, or by Saut and Xu (\cite{sautxu}) for a  model of full dispersion. In the Boussinesq-Peregrine case, a short study (see later Section \ref{c3:section04} for more details) leads us to introduce the following modified equation: 
\begin{equation}
\left\{
\begin{aligned}
&\dt\zeta+\nablag\cdot(h\Vu)=0\\
&(I+\mu(\Tb-\nablag(\hbi\nablag\cdot(h_b\cdot))-\hbi\nablago\nablago\cdot))\dt\Vu+\epsilon\Vu\cdot\nablag\Vu+(I-\mu\nablag\hbi\nablag\cdot(h_b\cdot))\nablag\zeta=0
\end{aligned}\right. \label{c3:boussinesq_modified_faux1}
\end{equation}

The main result of this paper is the following (see later Theorem \ref{c3:large_time_theorem} for a precise statement):

\begin{theorem}
\begin{enumerate}[label = ( \arabic*)]
\item The equation $ \eqref{c3:boussinesq_modified_faux1} $ is locally well-posed on a time interval $[0;T]$ where $T$ only depends on the initial data (and not on $\mu$), in dimension $d=1,2$.
\item In dimension $1$, the equations \eqref{c3:boussinesq_modified_faux1} admit a unique solution on a time interval of the form $[0;\frac{T}{\epsilon}]$ where $T$ only depends on the initial data.
\end{enumerate}
\end{theorem}
\begin{remark}
For technical reasons which are discussed further below, the Theorem we prove is only true in dimension $1$. However, we precisely explain in this paper the difficulties raising for a proof in dimension $2$.
\end{remark}
Though this Theorem is proved by adapting the technique used by Bresch-Métivier in \cite{bresch_metivier}, its adaptation to the case of a dispersive equation has not been done yet in the literature to our knowledge. As one shall see later in Section \ref{c3:section03}, this result is tied to a singular perturbation problem.
The plan of the article is the following:
\begin{itemize}[label=--,itemsep=0pt] \item In Section \ref{c3:section01}, we prove a local existence result for the Boussinesq-Peregrine equation in dimension $d=1,2$,
\item In Section \ref{c3:section04} we introduce a modified Boussinesq-Peregrine equation and in Section \ref{c3:section02} we prove its local well-posedness in dimension $d=1,2$,
\item In Section \ref{c3:section03}, we prove the long time existence result for the modified Boussinesq-Peregrine equation, in dimension $d=1$.
\end{itemize}

\subsection{Notations}\label{c3:notations}
We introduce here all the notations used in this paper.
 \subsubsection{Operators and quantities} Because of the use of dimensionless variables (see before the "dimensionless equations" paragraph), we use the following twisted partial operators: 
\begin{align*}
&\nabla^{\gamma} = {}^t(\partial_x,\gamma\partial_y) \quad &\text{ if } d=2\quad &\text{ and } &\nabla^{\gamma} = \partial_x &\quad \text{ if } d=1, \\
&\nablago = {}^t(-\gamma\partial_y,\partial_x) &\text{ if } d=2\quad &\text{ and } &\nablago=0 &\quad \text{ if } d=1. 
\end{align*}
\begin{remark}All the results proved in this paper do not need the assumption that the typical wave lengths are the same in both directions, ie $\gamma = 1$. However, if one is not interested in the dependence of $\gamma$, it is possible to take $\gamma = 1$ in all the following proofs. A typical situation where $\gamma\neq 1$ is for weakly transverse waves for which $\gamma=\sqrt{\mu}$; this leads to weakly transverse Boussinesq systems and the Kadomtsev–Petviashvili equation (see \cite{lannes_saut}). A byproduct of our results is therefore a generalization to the case of nonflat bottoms of the results on weakly transverse Boussinesq systems of \cite{lannes_saut}; this opens new perspectives towards the derivation and justification of Kadomtsev–Petviashvili equations for nonflat bottoms in the spirit of \cite{samer} for the KdV equation.
\end{remark}
We define $a\vee b$ for two real numbers $a,b$ by: $$a\vee b = \max(a,b).$$
For all $\alpha =(\alpha_1,...,\alpha_d)\in\mathbb{N}^d$, we define $\partial^\alpha$ the operator of $\mathcal{S}'(\R^d)$ by:
$$\partial^\alpha = \partial_{x_1}^{\alpha_1}...\partial_{x_d}^{\alpha_d}.$$
 We use the classical Fourier multiplier 
$$\Lambda^s = (1-\Delta)^{s/2} \text{ on } \mathbb{R}^d$$ defined by its Fourier transform as $$\mathcal{F}(\Lambda^s u)(\xi) = (1+\vert\xi\vert^2)^{s/2}(\mathcal{F}u)(\xi)$$ for all $u\in\mathcal{S}'(\mathbb{R}^d)$. We also use the following operators:
$$\forall\Vu,\qquad \Tb\Vu = -\frac{1}{3h_b}\nablag(h_b^3\nablag\cdot\Vu)+\frac{\beta}{2h_b}[\nablag(h_b^2\nablag b\cdot\Vu)-h_b^2\nablag b\nablag\cdot\Vu]+\beta^2\nablag b\nablag b\cdot\Vu$$ in all this paper, and the operators:
$$\A=\nablag(\hbi\nablag\cdot(h_b\cdot)),\qquad \Bi = (I+\mu\Tb-\mu\nablag(\hbi\nablag\cdot(h_b\cdot))-\mu\hbi\nablago\nablago\cdot)$$ in Section \ref{c3:section04}.

\subsubsection{Functional spaces}
The standard scalar product on $L^2(\mathbb{R}^d)$ is denoted by $(\quad,\quad)_2$ and the associate norm $\vert\cdot\vert_2$. We will denote the  standard scalar product on Sobolev spaces  $H^s(\R^d)$  by $(\quad,\quad)_{H^s}$ and the associate norm by $\vert \cdot\vert_{H^s}$. We denote the norm $W^{k,\infty}(\R^d)$ by  $\vert \cdot\vert_{W^{k,\infty}},$ and we use the notation $  \vert\cdot\vert_\infty=\vert\cdot\vert_{W^{0,\infty}}$ when no ambiguity is possible. \par\vspace{\baselineskip} We also introduce in Section \ref{c3:section03} for all $s\in\R$ the Banach space $$X^s(\R^d) = \lbrace f\in L^2(\R^d)^d,\vert f\vert_{X^s}<\infty\rbrace$$ endowed with the norm $$\vert f\vert_{X^s} = \vert f\vert_{X^s}^2 = \vert f\vert_{H^s}^2+\mu\vert\nablag\cdot f\vert_{H^s}^2.$$

\section{Local existence for the Boussinesq-Peregrine equation}\label{c3:section01}
In this section, we prove a local existence result for the Boussinesq-Peregrine equation in dimension $d=1,2$, on a time interval independent on $\mu$. We recall that the Boussinesq-Peregrine equation of unknowns $\Vu$ and $\zeta$ is:
\begin{align}
\begin{cases}
\partial_t\zeta+\nablag\cdot(h\Vu)=0 \\
[1+\mu\Tb]\partial_t\Vu+\epsilon(\Vu\cdot\nablag)\Vu+\nablag\zeta = 0 \label{c3:boussinesq_equation}
\end{cases}
\end{align} where $h=1+\epsilon\zeta-\beta b$ and where $\Tb$ is the following operator : 
\begin{equation}\forall\Vu,\qquad \Tb\Vu = -\frac{1}{3h_b}\nablag(h_b^3\nablag\cdot\Vu)+\frac{\beta}{2h_b}[\nablag(h_b^2\nablag b\cdot\Vu)-h_b^2\nablag b\nablag\cdot\Vu]+\beta^2\nablag b\nablag b\cdot\Vu,\label{c3:deftb}\end{equation} with the notation  $h_b = 1-\beta b$, where $b$ describes the variations of the bottom and is known. We also recall (see section \ref{c3:notations}) the notation $$\vert f\vert_{X^s} =  \vert f\vert_{H^s}^2+\mu\vert\nablag\cdot f\vert_{H^s}^2.$$ Let us state a local existence result for the Boussinesq-Peregrine equation:

\begin{theorem}
Let $t_0>d/2$ and $s>t_0+1$. Let $b\in H^{s+2}(\R^d)$ be such that there exists $h_{\min}>0$ such that $$\underset{X\in\R^d}{\inf} h_b(X) \geq h_{\min}.$$ Let $\epsilon$, $\beta$ be such that $$0\leq \epsilon,\beta\leq 1.$$ Let $U_0=(\zeta_0,\Vu_0)\in H^s(\R^d)\times X^s(\R^d)$.  Then, there exists $\mu_{\max}>0$ such that for all $0\leq\mu\leq\mu_{\max}$ with $$ \epsilon = O(\mu),$$ there exists $T^*>0$ and a unique solution $U=(\zeta,\Vu)\in C([0;T^*[;H^s(\R^d)^{d+1})$ of the equation \eqref{c3:boussinesq_equation} with initial condition $U(0) = U_0$.\par\vspace{\baselineskip}
 
 Moreover, for all $T<T^*$, if one chooses $$\lambda \geq    \underset{t\in[0;T]}{\sup} C_3(\vert U\vert_{W^{1,\infty}},\vert h_b\vert_{H^{s+2}},\mu_{\max})(t),$$ the solution $U$ satisfies the following inequality:
\begin{equation}\forall  t\leq T, \quad \mathcal{E}(t)  \leq\frac{\lambda}{C_1(h_{\min})} \int_0^t  e^{\frac{\lambda}{C_1(h_{\min})}(t-t')} C_4(\vert U\vert_{H^s})(t')dt'+ \frac{C_2(\vert h_b\vert_{H^{t_0}})}{C_1(h_{\min})}\mathcal{E}(0)e^{\frac{\lambda}{C_1(h_{\min})}t}\label{c3:estimate_sol}\end{equation}
with \begin{equation}\mathcal{E}(t) = (\mu\vert \nablag\cdot \Vu\vert_{H^s}^2+\vert U\vert_{H^s}^2)(t)\label{c3:energy_theorem}\end{equation} and where $C_i$ are non decreasing smooth functions of their arguments, for $i=1,2,3,4.$\label{c3:boussi_theorem}
\end{theorem}

\begin{remark} \label{c3:remark_estimate}\begin{itemize}[label=--,itemsep=0pt] \item It is very important to note that the energy estimate \eqref{c3:estimate_sol} implies that while $\vert U\vert_{W^{1,\infty}}(t)$ is bounded, the solution $U$ can be continued. More precisely, if one sets $$T = \sup\lbrace t, U \text{ exists on } C([0;t[;H^s(\R^d)^{d+1})\rbrace$$ then if $T<\infty$, one has $\underset{t\rightarrow T}{\sup} \vert U\vert_{W^{1,\infty}}(t) = +\infty$. Moreover, one has to notice that the energy $ \mathcal{E}$ defined in the statement of the Theorem by \eqref{c3:energy_theorem} controls $H^s$ norms of both $U$ and $\nablag\cdot \Vu$, while the energy estimate \eqref{c3:estimate_sol} only requires a  bound of $\mathcal{E}$  by the $H^s$ norm of $U$. In particular, it suffices to bound the $H^s$ norm of $U$ (instead of $U$ and $\nablag\cdot \Vu$) to use a continuity argument.
\item Note that the time of existence $T^*$ defined by Theorem \ref{c3:boussi_theorem} is independent on $\mu$. This is crucial in view of the proof of the consistency of the Boussinesq-Peregrine equation with the Water-Waves equations (see Theorem \ref{c3:consistance_def}).
\end{itemize}\end{remark}

\begin{proof}[Proof of Theorem \ref{c3:boussi_theorem}]

The system \eqref{c3:boussinesq_equation} can be put under the form 
 
 \begin{equation*}
B \partial_t U + \sum_{j=1}^d A_j(U) \partial_j U = 0
 \end{equation*}
 where \begin{equation}
 B = \begin{pmatrix}
 1 &0 \\
 0 & I+\mu\Tb
 \end{pmatrix},\qquad A_j(U) = \begin{pmatrix}
  \epsilon\Vu_j & h \\
 1 &\epsilon \Vu_jI_d
 \end{pmatrix}\qquad \forall 1\leq j\leq d.\label{c3:defmatrix}
  \end{equation}
The non linear terms of the equation can be symmetrized if we multiply the system by 
\begin{equation*}
\tilde{S}(U) = \begin{pmatrix}
1 &0 \\
0 &hI_d
\end{pmatrix}
\end{equation*} but for the reasons explained in Section \ref{c3:intro_explication} we use rather the following matrix:
\begin{equation}
S = \begin{pmatrix}
1 &0 \\
0 &h_b I_d
\end{pmatrix}.\label{c3:defsymmetrizer}
\end{equation}
Using this symmetrizer brings one difficulty: the operator $h_b(I+\mu\Tb)$ is symmetric, but multiplying the second equation of \eqref{c3:boussinesq_equation} by $h_b$ does not symmetrize the non linear terms of the form $A_j(U)\partial_j U$ defined by \eqref{c3:defmatrix}, for $j=1,..,d$.  The proof of Theorem \ref{c3:boussi_theorem} is inspired of the classical existence result for quasilinear hyperbolic systems (see \cite{taylor3} Chapter XVI for instance). We follow the following steps:\begin{itemize}[label=--,itemsep=0pt]
\item Step 1 : We solve a smoothed equation involving a mollifier $(1-\delta\Delta)$.
\item Step 2 : We prove that the existence time of the solution of the mollified equation does not depend on $\delta$, and the uniform bound in $H^s$ norm of this solution.
\item Step 3 : We pass to the limit $\delta$ goes to zero in the mollified equation to get a solution of the equation \eqref{c3:boussinesq_equation}.
\item Step 4 : We recover regularity for the solution of \eqref{c3:boussinesq_equation}.\end{itemize}

\textbf{Step 1}\qquad We solve the Cauchy problem  \begin{align}
\begin{cases}
(1-\delta\Delta)SB(1-\delta\Delta)\partial_t \Ud + \sum_{i=1}^d SA_j(\Ud)\partial_j \Ud = 0\label{c3:regularisee} \\
\Ud(0) = (1-\delta\Delta)^{-1} U_0,
 \end{cases}
\end{align}
of unkown $\Ud$ in the Banach space $H^s(\R^d)^{d+1}$. Recall that using the definition of $B$ given in \eqref{c3:defmatrix} and $S$ given by \eqref{c3:defsymmetrizer}, one has \begin{equation*} SB = \begin{pmatrix}
1 &0 \\
0 &h_b (I_d+\mu\Tb)
\end{pmatrix}.\end{equation*} In order to apply the Cauchy-Lipschitz Theorem, one must check that the application 
\begin{displaymath} \left. \begin{array}{rcl}
&H^s(\R^d)^{d+1} &\longrightarrow H^s(\R^d)^{d+1} \\
&U &\longmapsto (1-\delta\Delta)^{-1}(SB)^{-1}(1-\delta\Delta)^{-1} \sum_{i=1}^d A_j(U) \partial_j U
\end{array}\right.\end{displaymath}
is well defined and locally Lipschitz. The unique difficulty is to check that $(h_b(I+\mu\Tb))^{-1}$ is well defined from $H^{s}$ to $H^s$. It is the point of the following Proposition (see \cite{david} Chapter 5 Lemma 5.44 for a full proof). We first define the Banach space $$X^s = \lbrace V\in H^s(\R^d)^{d}, \nablag\cdot V\in H^s(\R^d) \rbrace$$ endowed with the norm $$\vert V\vert_{X^s}^2 = \vert V\vert_{H^s}^2+\mu\vert \nablag\cdot V\vert_{H^s}^2.$$

\begin{proposition}
Let $t_0>d/2$, $\beta\leq 1$ and $b\in H^{t_0+1}(\R^d)$ be such that there exists $h_{\min}$ such that $h_b=1-\beta b \geq h_{\min}$. Then the mapping $$h_b(I+\mu\Tb):X^0\longrightarrow L^2(\R^d)^d+\nablag L^2(\R^d)$$ is well defined, one-to-one and onto. 	One has, for  all $V\in X^0$: 
\begin{equation*}
C_1(h_{\min})\vert V\vert_{X^0}^2 \leq (h_b(I+\mu\Tb)V,V)_2\leq C_2(\vert h_b\vert_{H^{t_0+1}})\vert V\vert_{X^0}^2
\end{equation*}
   where $C_i$ are non decreasing functions of its arguments.  Moreover, for all $s\geq 0$, if $b\in H^{1+s\vee t_0}(\R^d)$, then $$\forall  W\in H^s(\R^d)^d, \vert (h_b(I+\mu\Tb))^{-1} W\vert_{X^s} \leq C(\frac{1}{h_{\min}},\vert b\vert_{H^{1+s\vee t_0}}) \vert W\vert_{H^s}$$ where $C$ is a non decreasing function of its arguments. Moreover, one has, for all $s\in\R$:

$$\sqrt{\mu}\vert (h_b(I+\mu\Tb))^{-1}\nablag W\vert_{H^s} \leq C(\frac{1}{h_{\min}},\vert b\vert_{H^{1+\vert s\vert\vee t_0}})\vert W\vert_{H^s}.$$

\label{c3:regu_inverse}\end{proposition}

Therefore, the Cauchy-Lipschitz Theorem applies and the equation \eqref{c3:regularisee} has a unique solution $\Ud\in C([0;T^\delta[;H^s(\R^d)^{d+1})$, and if $T^\delta <+\infty$ one has $$\underset{t\rightarrow T^\delta}{\lim} \vert \Ud(t)\vert_{H^s} = +\infty.$$
   
\textbf{Step 2}\qquad We now check that one can choose $T^\delta$ independent of $\delta$ by comparing $\vert \Ud(t)\vert_{H^s}$ with a solution of an ODE independent of $\delta$, and using a Gronwall Lemma.  We define $$\Ud_s = \Lambda^s \Ud.$$ The unknown $\Ud_s$ satisfies the following system :
\begin{equation}
(1-\delta\Delta)SB(1-\delta\Delta)\partial_t \Ud_s + \sum_{j=1}^d SA_j(\Ud)\partial_j \Ud_s = F\label{c3:quasiboussi}
\end{equation}
where we wrote the commutators under the form 
\begin{equation}
F = (1-\delta\Delta)[SB,\Lambda^s](1-\delta\Delta)\partial_t \Ud+\sum_{j=1}^d [SA_j(\Ud),\Lambda^s]\partial_j \Ud.\label{c3:inter2}
\end{equation}
In order to estimate, $\vert\Ud(t)\vert_{H^s}$, recall that \begin{equation*} SB = \begin{pmatrix}
1 &0 \\
0 &h_b (I_d+\mu\Tb)
\end{pmatrix}\end{equation*} and remark, using Proposition \eqref{c3:regu_inverse}, that $$\vert (1-\delta\Delta) V^\delta_s\vert_{X^0}\sim (h_b(I+\mu\Tb)(1-\delta\Delta)V^\delta_s,(1-\delta\Delta)V^\delta_s)_2$$ where the implicit constant only depend on $h_b$. Therefore, one computes:
   \begin{align*}
  \frac{d}{dt}\frac{1}{2}((1-\delta\Delta)S B(1-\delta\Delta) \Ud_s,\Ud_s)_2 &= ((1-\delta\Delta)SB(1-\delta\Delta)\partial_t \Ud_s,\Ud_s)_2.
   \end{align*}
Note that the symmetry of $SB$, and more precisely of $h_b(I+\mu\Tb)$ is crucial here. One uses the equation \eqref{c3:quasiboussi} to replace $(1-\delta\Delta)SB(1-\delta\Delta)\partial_t \Ud_s$ by its expression. One gets:
   \begin{align}
  \frac{d}{dt}\frac{1}{2}((1-\delta\Delta)S B(1-\delta\Delta) \Ud_s,\Ud_s)_2 &= -\sum_{j=1}^d(S A_j(\Ud)\partial_j \Ud_s,\Ud_s)_2+(F,\Ud_s)_2. \label{c3:derivee_energie}
   \end{align}
Let us check that the first term of the rhs of \eqref{c3:derivee_energie} has a contribution of order zero to the energy estimate. One uses the definition of $A_j$ given by \eqref{c3:defmatrix} to put this matrix under the form $A_j(\Ud)=\tilde{A_j}(\Ud)+C(\Ud)$ with
\begin{equation}\tilde{A_j}   (\Ud) = \begin{pmatrix}
\epsilon\Vu_j^\delta &h_b \\
1 &\epsilon\Vu_j^\delta I_d
\end{pmatrix},\qquad C(\Ud) = \begin{pmatrix}
0 & \epsilon\zeta^\delta \\
0 &0
\end{pmatrix},\label{c3:defmatrixtilde}\end{equation}
for $j=1,..,d$. Note that since $S$ is not a true symmetrizer for the equation \eqref{c3:boussinesq_equation}, the matrix $SA_j$ is not symmetric. The above decomposition allows us to write $SA_j$ under the form of a symmetric matrix ($S\tilde{A_j}$) plus a rest which we intend to control in the energy estimates. We now write:  \begin{equation}\sum_{j=1}^d(S A_j(\Ud)\partial_j \Ud_s,\Ud_s)_2   = \sum_{j=1}^d(S C(\Ud)\partial_j \Ud_s,\Ud_s)_2+\sum_{j=1}^d(S \tilde{A_j}(\Ud)\partial_j \Ud_s,\Ud_s)_2.\label{c3:symmetric_term}\end{equation} Using the definition of $C(\Ud)$ given by \eqref{c3:defmatrixtilde}, the first term  of the rhs of \eqref{c3:symmetric_term} is equal to $\ds \int \zeta^\delta \zeta_s^\delta\epsilon \nablag\cdot\Vu_s^\delta$ and therefore is controlled, using the Cauchy-Schwarz inequality by \begin{equation}\ds \vert  \int \epsilon \zeta^\delta\nablag\cdot\Vu_s^\delta\zeta_s^\delta\vert \leq \vert \zeta^\delta\vert_{W^{1,\infty}}\vert \mu\nablag\cdot\Vu_s^\delta\vert_2\vert\zeta_s^\delta\vert_2 \label{c3:constant1}\end{equation}where we used the Boussinesq regime condition $$\epsilon\leq C\mu$$ stated by the Theorem. For the second term of the rhs of \eqref{c3:symmetric_term}, one can write, for all $1\leq j\leq d$, and using the symmetry of $S\tilde{A_j}(\Ud)$ (recall the definition of the symmetrizer $S$ given by \eqref{c3:defsymmetrizer}):  
   \begin{align*}
   (S\tilde{A_j}( \Ud)\partial_j \Ud_s,\Ud_s)_2 &= (\partial_j \Ud_s, S\tilde{A_j}(\Ud) \Ud_s)_2 \\
   &= -\big(\Ud_s,\partial_j (S\tilde{A_j}(\Ud)\Ud_s)\big)_2
   \end{align*}
   by integrating by parts. Now, one has 

   $$-\big(\Ud_s,\partial_j (S\tilde{A_j}(\Ud)\Ud_s)\big)_2 = -\big(\Ud_s,\partial_j(S\tilde{A_j}(\Ud))\Ud_s\big)_2-(\Ud_s,S\tilde{A_j}(\Ud) \partial_j \Ud_s)_2$$
and thus one has 
   \begin{equation}(S\tilde{A_j} (\Ud)\partial_j \Ud_s,\Ud_s)_2 = -\frac{1}{2}\big(\Ud_s,\partial_j(S\tilde{A_j}(\Ud))\Ud_s\big)_2\label{c3:inter1}.\end{equation} Using the definition of $S$ given by $\eqref{c3:defsymmetrizer}$, and the definition of $\tilde{A}_j$ given by \eqref{c3:defmatrixtilde}, one has $$S\tilde{A}_j=\begin{pmatrix}
   \epsilon\Vu^\delta_j &h_b \\
   h_b &\epsilon h_b\Vu^\delta_j I_d
\end{pmatrix}$$ for all $j=1,..,d$    and thus \eqref{c3:inter1} is controlled, using Cauchy-Schwarz inequality:
   \begin{equation}\vert(S\tilde{A_j} \Ud\partial_j \Ud_s,\Ud_s)_2 \vert\leq  \vert \Ud_s\vert_2^2 c_2(\vert \Ud\vert_{W^{1,\infty}},\vert h_b\vert_{W^{1,\infty}}),\label{c3:constant2}\end{equation} where $c_2$ is a non decreasing and smooth function of its arguments.	\par\vspace{\baselineskip}
   
   We now control the second term of the right hand side of the energy estimate \eqref{c3:derivee_energie}. Using the definition of $F$ given by \eqref{c3:inter2}, one has that $$(F,\Ud_s)_2  = A_3+A_4$$ where $$A_3 = ((1-\delta\Delta)[SB,\Lambda^s](1-\delta\Delta)\partial_t \Ud,\Ud_s)_2$$ and $$A_4 = ([SA_j(\Ud),\Lambda^s]\partial_j \Ud,\Ud_s)_2.$$

\textit{- Control of $A_3$}\qquad We start by  replacing $(1-\delta\Delta)\partial_t \Ud$ by its expression given in the equation \eqref{c3:regularisee}:
\begin{equation}\begin{aligned}
&((1-\delta\Delta)[SB,\Lambda^s](1-\delta\Delta)\partial_t \Ud,\Ud_s)\\ =&-\sum_{j=1}^d ((1-\delta\Delta)[SB,\Lambda^s](SB)^{-1}(1-\delta\Delta)^{-1}SA_j(\Ud)\partial_j\Ud,\Ud_s)_2.
\end{aligned}\label{c3:inter3}\end{equation}
One has to control this term uniformly  with respect to $\delta$, and deals with the fact that $(h_b(I+\mu\Tb))^{-1}$ is not optimally estimated. More precisely, that is absolutely not clear that\footnote{The operator $h_b(I+\mu\Tb)$ is not technically elliptic of order $1$, since its inverse only controls the divergence (and not a full derivative). This is actually a big issue for all the local existence results for the Boussinesq-Peregrine equation \eqref{c3:boussinesq_equation}. This is also the reason why a Nash-Moser scheme must be used to solve the Green-Naghdi equations in $2d$ (see \cite{alvarez2007nash}).} $$\vert h_b(I+\mu\Tb)f(h_b(I+\mu\Tb))^{-1} u\vert_{{H^s}} \leqslant \vert u\vert_{{H^s}}$$ for a smooth function $f$. One has to recall that $$
 B = \begin{pmatrix}
 1 &0 \\
 0 & I+\mu\Tb
 \end{pmatrix}$$ so that 	
 $$
 [SB,\Lambda^s] = \begin{pmatrix}
 0 &0 \\
 0 & [h_b(I+\mu\Tb),\Lambda^s]
 \end{pmatrix}.$$  Using the definition of $\Tb$ given by \eqref{c3:deftb}, one writes this operator under the form:
 $$I+\mu\Tb = I+\mu(A+B+C+D),$$
where \begin{equation}A = -\nablag(h_b^3\nablag\cdot),\quad B = \beta\nablag(h_b^2\nablag b\cdot),\qquad C=-\beta h_b^2\nablag b\nablag\cdot,\qquad D=\beta^2\nablag b\nablag\cdot \label{c3:defoperateur}\end{equation}
One expands the commutator \eqref{c3:inter3} with respect to $A,B,C,D$. We set $$A_{31} = ((1-\delta\Delta)\mu[A,\Lambda^s](h_b(I+\mu\Tb))^{-1}(1-\delta\Delta)^{-1}SA_j(\Ud)\partial_j\Vde,\Vde_s)_2$$ and we now control this latter term. For all $1\leq j\leq d$, using the definition of $A$ given by \eqref{c3:defoperateur}, one has, integrating by parts:
\begin{align*}
A_{31}=((1-\delta\Delta)[h_b^3,\Lambda^s]\sqrt{\mu}\nablag\cdot(h_b(I+\mu\Tb))^{-1}(1-\delta\Delta)^{-1}SA_j(\Ud)\partial_j\Vde,\sqrt{\mu}\nablag\cdot\Vde_s)_2.
\end{align*}
Using Cauchy-Schwarz inequality and splitting $(1-\delta\Delta)$, one gets:
\begin{align*}
\vert A_{31}\vert \leq &\sqrt{\mu}\vert \nablag\cdot \Vde_s\vert_2\big(\vert [h_b^3,\Lambda^s]\sqrt{\mu}\nablag\cdot(h_b(I+\mu\Tb))^{-1}(1-\delta\Delta)^{-1} SA_j(\Ud)\partial_j\Vde\vert_2\\ +&\delta\vert [h_b^3,\Lambda^s]\sqrt{\mu}\nablag\cdot(h_b(I+\mu\Tb))^{-1}(1-\delta\Delta)^{-1} SA_j(\Ud)\partial_j\Vde\vert_{H^2} \big).
\end{align*}
We now use the Kato-Ponce estimate of Proposition \ref{c3:katoponce} to control $[h_b^3,\Lambda^s]$, using the fact that $s>d/2+1$ (and thus $H^{s-1}(\R^d)$ and $H^s(\R^d)$ are respectively continuously injected into $L^\infty(\R^d)$ and $W^{1,\infty}(\R^d)$):
\begin{equation}\begin{aligned}
\vert A_{31}\vert\leq &\sqrt{\mu}\vert \nablag\cdot \Vde_s\vert_2C_s(\vert \nablag h_b\vert_{H^{s+1}})\big(\vert\sqrt{\mu}\nablag\cdot(h_b(I+\mu\Tb))^{-1}(1-\delta\Delta)^{-1} SA_j(\Ud)\partial_j\Vde\vert_{H^{s-1}} \\
+&\delta\vert\sqrt{\mu}\nablag\cdot(h_b(I+\mu\Tb))^{-1}(1-\delta\Delta)^{-1} SA_j(\Ud)\partial_j\Vde\vert_{H^{s+1}}\big),
\end{aligned}\label{c3:inter4}\end{equation}
where $C_s$ is a smooth non decreasing function which only depends on $s$. We now control the operator $\sqrt{\mu}\nablag\cdot (h_b(I+\mu\Tb))^{-1}$ in $H^{s-1}$ and $H^{s+1}$ norms by using the last part of Proposition \ref{c3:regu_inverse} and a duality argument. One has, for all $k\geq t_0$ and all $u\in H^k(\R^d)^d$, using the symmetry of $h_b(I+\mu\Tb)$ :
\begin{equation}\begin{aligned}
\sqrt{\mu}\vert \nablag\cdot h_b(I+\mu\Tb)^{-1} u\vert_{H^k} &= \sqrt{\mu}\underset{v\in H^{-k}(\R^d)\atop\vert v\vert_{H^{-k}}=1}{\sup} (\nablag\cdot (h_b(I+\mu\Tb))^{-1} u,v)_2 \\
&= \sqrt{\mu}\underset{v\in H^{-k}(\R^d)\atop\vert v\vert_{H^{-k}}=1}{\sup} -( u,(h_b(I+\mu\Tb))^{-1}\nablag v)_2 \\
&\leq \underset{v\in H^{-k}(\R^d)\atop\vert v\vert_{H^{-k}}=1}{\sup}  C(\frac{1}{h_{\min}},\vert  b\vert_{H^{k+1}})\vert u\vert_{H^k}\vert v\vert_{H^{-k}}\\
&\leq C(\frac{1}{h_{\min}},\vert b\vert_{H^{k+1}})\vert u\vert_{H^k},
\end{aligned}\label{c3:inter5}\end{equation}
where $C$ is a smooth non decreasing function of its arguments, and where we used the fact that $k\geq t_0$. Using \eqref{c3:inter5} with $k=s-1$ and $k=s+1$ in \eqref{c3:inter4} (recall that $s>t_0+1$), one gets:
\begin{align*}
\vert A_{31}\vert&\leq C\sqrt{\mu}\vert \nablag\cdot \Vde_s\vert_2C(\frac{1}{h_{\min}},\vert   b\vert_{H^{s+2}})\big(\vert (1-\delta\Delta)^{-1}SA_j(\Ud)\partial_j\Vde\vert_{H^{s-1}}\\
&+\delta\vert (1-\delta\Delta)^{-1}SA_j(\Ud)\partial_j\Vde\vert_{H^{s+1}}\big) \\
&\leq  C\sqrt{\mu}\vert \nablag\cdot \Vde_s\vert_2C(\frac{1}{h_{\min}},\vert  b\vert_{H^{s+2}})\big(\vert SA_j(\Ud)\partial_j\Vde\vert_{H^{s-1}}+C\delta \frac{1}{\delta}\vert SA_j(\Ud)\partial_j\Vde\vert_{H^{s-1}}\big)
\end{align*}
where we used the estimates $\vert (1-\delta\Delta)^{-1} f\vert_{H^{s-1}} \leq \vert f\vert_{H^{s-1}}$ and $\vert (1-\delta\Delta)^{-1} f\vert_{H^{s+1}} \leq \frac{C}{\delta}\vert f\vert_{H^{s-1}}$ with $C$ independent on $\delta$. We recall that $$SA_j(\Ud) = \begin{pmatrix}
\epsilon\Vu^\delta_j &h \\
h_b &\epsilon h_b\Vu^\delta_j
\end{pmatrix}$$ for $j=1,..,d$ and we use the Moser estimate of Proposition \ref{c3:moser} and the fact that $s-1>d/2$ to conclude:
\begin{align*}
\vert A_{31}\vert\leq C(\frac{1}{h_{\min}},\vert  b\vert_{H^{s+2}})\sqrt{\mu}\vert \nablag\cdot\Vde_{s}\vert_2\vert \Vde_s\vert_2^2
\end{align*}
with $C$ a smooth, non decreasing function of its arguments.

To control the term of \eqref{c3:inter3} involving $B$, using the definition of $B$ given by \eqref{c3:defoperateur}, we write, integrating by parts:
\begin{align*}
\sum_{j=1}^d&((1-\delta\Delta)\mu[B,\Lambda^s](h_b(I+\mu\Tb))^{-1}(1-\delta\Delta)^{-1}SA_j(\Ud)\partial_j\Vde,\Vde_s)_2 \\
&= -\mu((1-\delta\Delta)[h_b^3\nablag b\cdot,\Lambda^s](h_b(I+\mu\Tb))^{-1}(1-\delta\Delta)^{-1}SA_j(\Ud)\partial_j\Vde,\nablag\cdot\Vde_s)_2
\end{align*}
and we use exactly the same techniques as used for the control of $A_{31}$ to get the same control. The terms of \eqref{c3:inter3} involving $C$ and $D$ are easily controlled by 
$$	 C(\frac{1}{h_{\min}},\vert  b\vert_{H^{s+2}})\sqrt{\mu}\vert\Vde_s\vert_2^3.$$
We finally proved that:
 \begin{equation} A_3\leq c_3 (\vert \Ud\vert_{W^{1,\infty}}\vert,\vert \nablag b\vert_{H^{s+2}})\vert \Ud_s\vert_2^2(\vert\Ud_s\vert_2+\sqrt{\mu}\vert\nablag\cdot\Vde_s\vert_2)\label{c3:constant3}
\end{equation} where $c_3$ is a smooth positive non decreasing function of its argument and independent of $\delta$.\par\vspace{\baselineskip}

\textit{- Control of $A_4$}\qquad

Let us now control $A_4$ by using the Kato-Ponce estimate of Proposition \ref{c3:katoponce} and the Moser estimate of Proposition \ref{c3:moser}, using again that $s>d/2+1$: 
   \begin{align}
   ([SA_j(\Ud),\Lambda^s]\partial_j \Ud,\Ud_s)_2 &\leq  C(\vert \nablag  SA_j(\Ud)\vert_{H^{s-1}}\vert\partial_j \Ud\vert_{L^\infty}+\vert\nablag SA_j(\Ud)\vert_{L^\infty} \vert \partial_j \Ud\vert_{H^{s-1}})\vert \Ud_s\vert_2 \nonumber\\
   &\leq (C(\vert \Ud\vert_\infty)\vert \Ud\vert_{H^s} \vert \partial_j \Ud\vert_\infty\vert h_b\vert_{H^s}+C(\vert\Ud\vert_{W^{1,\infty}},\vert h_b\vert_{W^{1,\infty}}) \vert\Ud\vert_{H^s}) \vert \Ud_s\vert_2 \nonumber\\
   &\leq c_4(\vert \Ud\vert_{W^{1,\infty}},\vert h_b\vert_{H^s}) \vert \Ud_s\vert_2^2 .\label{c3:constant4}
   \end{align}

If one puts together \eqref{c3:constant1}, \eqref{c3:constant2}, \eqref{c3:constant3} and \eqref{c3:constant4}, one gets : 
\begin{align}
\frac{d}{dt}((1-\delta\Delta)SB (1-\delta\Delta)  \Ud_s,\Ud_s)_2 \leq c_5(\vert U^\delta\vert_{W^{1,\infty}},\vert h_b\vert_{H^{s+2}},\mu_{\max})\big(\vert U_s^\delta\vert_2^2 + \mu\vert\nablag\cdot \Vde_s\vert_2^2 +F(\vert \Ud_s\vert_2)\big)  \label{c3:inequality_precise}
\end{align}
with $c_5$ and $F$ some smooth non decreasing functions of their arguments, independent of $\delta$.


At this point, recalling the equivalence $\vert V\vert_{X^0}\sim (h_b(I+\mu\Tb)V,V)_2$ stated by Proposition \ref{c3:regu_inverse},  we proved that  $$\frac{d}{dt}( (1-\delta\Delta)SB (1-\delta\Delta)  \Ud_s,\Ud_s)_2 \leq F( (1-\delta\Delta)SB (1-\delta\Delta)  \Ud_s,\Ud_s)_2)$$ where $F$ is a Lipschitz function which does not depend on $\delta$. By Cauchy-Lipschitz theorem, there exists $T^*>0$ such that the Cauchy problem 
\begin{align*}
\begin{cases}
\frac{d}{dt}g(t) &= F(g(t)) \\
g(0) &= \vert U(0)\vert_{H^s}
\end{cases}
\end{align*}
admits a unique solution $g$ on a time interval $[0;T^*]$. By Gronwall's lemma, one has for all $t<T^*$ that 

$$( (1-\delta\Delta)SB (1-\delta\Delta)  \Ud_s,\Ud_s)_2 \leq g(t)$$ and consequently, using again the equivalence $\vert V\vert_{X^0}\sim (h_b(I+\mu\Tb)V,V)_2$ stated by Proposition \ref{c3:regu_inverse}:
\begin{equation}\forall 0\leq t\leq T^*,\qquad \vert (1-\delta\Delta)\Ud_s\vert_2^2+\mu\vert (1-\delta\Delta)\nablag\cdot\Vu^\delta_s\vert_2^2 \leq \frac{1}{C_1(h_{\min})} g(t).\label{c3:boussinesq_bound}\end{equation} This proves that the $H^s$ norm of $\Ud$ does not explode as $t$ goes to $T^*$, and then $T^\delta > T^*$, which give us a uniform time of existence for $\Ud$ independent of $\delta$.\par\vspace{\baselineskip}

We can be more precise for all $0<T<T^*$ if one chooses $$\lambda \geq  \underset{t\in[0;T]}{\sup} c_5(\vert U^\delta\vert_{W^{1,\infty}},\vert h_b\vert_{H^{s+2}},\mu_{\max})(t)$$ then one has the following inequality, using  estimate \eqref{c3:inequality_precise}:

\begin{equation*}
\frac{d}{dt}((1-\delta\Delta)SB (1-\delta\Delta) \Ud_s,\Ud_s)_2 \leq \lambda (\vert\Ud_s\vert_2^2 + \mu \vert\nablag\cdot \Vu^\delta_s\vert_2^2+F(\vert\Ud_s\vert_2))
\end{equation*}
and by integrating in time and using one last time the equivalence $\vert V\vert_{X^0}\sim  (h_b(I+\mu\Tb)V,V)_2$ stated by Proposition \ref{c3:regu_inverse}, one gets, for all $0\leq t\leq T$:
\begin{align*}
(\mu\vert \nablag\cdot (1-\delta\Delta)^{1/2}V^\delta_s\vert_2^2&+\vert(1-\delta\Delta)^{1/2}\Ud_s\vert_2^2)(t') \leq\frac{\lambda}{C_1(h_{\min})} \int_0^t   F(\vert\Ud_s\vert_2)(t')dt'\\
&+\frac{\lambda}{C_1(h_{\min})} \int_0^t (\vert(1-\delta\Delta)^{1/2}\Ud_s\vert_2^2+\mu\vert\nablag\cdot(1-\delta\Delta)^{1/2} \Vu^\delta_s\vert_2^2 )(t') dt'\\
&+\frac{C_2(\vert h_b\vert_{H^{t_0}})}{C_1(h_{\min})}\mathcal{E}(0)\end{align*} where we recall that the energy $\mathcal{E}$ is defined by \eqref{c3:energy_theorem}. One can conclude by Gronwall Lemma that:
\begin{align}(\mu\vert \nablag\cdot (1-\delta\Delta)^{1/2}V^\delta_s\vert_2^2+\vert(1-\delta\Delta)^{1/2}\Ud_s\vert_2^2)(t') & \leq\frac{\lambda}{C_1(h_{\min})} \int_0^t  e^{\frac{\lambda}{C_1(h_{\min})}(t-t')} F(\vert\Ud_s\vert_2)(t')dt' \nonumber \\&+ \frac{C_2(\vert h_b\vert_{H^{t_0}})}{C_1(h_{\min})}\mathcal{E}(0)e^{\frac{\lambda}{C_1(h_{\min})}t}. \label{c3:estimate_sol_delta}\end{align}

\textbf{Step 3-4}\qquad  The inequality \eqref{c3:boussinesq_bound} and the equation \eqref{c3:regularisee} prove that $ (\Ud)_\delta$ is bounded in the space\\$\ds L^\infty([0;T^*];H^s(\R^d))\cap W^{1,\infty}([0;T^*];H^{s-1})$. By compact embedding in $H^{s'}(\R^d)$ for all $s'<s$, one has the strong convergence of $ (\Ud)_\delta$ in $C([0;T^*];H^{s'}_{loc}(\R^d))$ to a function $U$. If one chooses $s'$ close enough to $s$, $H^{s'}(\R^d)$ is embedded in $C^1(\R^d)$ and one can pass to the limit in the non-linear terms of \eqref{c3:regularisee}. The linear terms do not raise any difficulty. It gives us a solution $U$ of the problem. A short analysis as in \cite{taylor3} Proposition XVI.1.4 shows that $U$ is in fact $C([0;T^*];H^s(\R^d))$. One can pass the limit $\delta$ goes to zero in the estimate \eqref{c3:estimate_sol_delta} and recovers the estimate \eqref{c3:estimate_sol} stated in the Theorem.$\qquad \Box $ 
\end{proof}

\section{Modified equation}\label{c3:section04}

As explained in the Introduction, the Boussinesq-Peregrine equation \eqref{c3:boussinesq_equation} does not have the proper structure to apply the technique used by Bresch-Métivier in \cite{bresch_metivier}. One time derivative of $\zeta$ is not equal to sum of terms of one space derivative  order of $\Vu$. \par\vspace{\baselineskip} 

In order to implement the technique used by \cite{bresch_metivier}, we modify a bit the equation without changing the consistency with the Water-Waves equation. More precisely, the Boussinesq-Peregrine equation is consistent at order $O(\mu^2)$ with the Water-Waves equation and therefore we look for a new equation consistent with the Boussinesq-Peregrine equation at a $O(\mu^2)$ order. In this new equation, one space derivative of $\zeta$ should have the "same order" as one time derivative of $\Vu$. For this purpose, we use the following formal consideration:
$$\partial_t \Vu = - (I+\mu\Tb)^{-1} (\epsilon\Vu\cdot\nablag\Vu+\nablag\zeta) = -\nablag\zeta +\mu R$$
where $R$ is of order $0$ or more in $\mu$ (recall that $\epsilon=O(\mu)$ in the Boussinesq-Peregrine regime). Therefore, one can add to the Boussinesq-Peregrine equation any expression of the form $\mu A\nablag\zeta$ and the corresponding term $\mu A\dt\Vu$, where $A$ is an operator independent on $\mu$, without changing the consistency: $$\mu A\partial_t\Vu = -\mu A\nablag\zeta + \mu^2R.$$ The operator $A$ should respect the following constraints:\begin{itemize}[label=--,itemsep=0pt]
\item the operator $h_bA$ should be symmetric, since multiplying the second equation by $h_b\Vu$ should give the time derivative of a positive quantity, such as $\vert\Vu\vert_2^2$ or $\vert\nablag\cdot\Vu\vert_2^2$;
\item the whole system must conserve a certain symmetry and be of the form: $B\dt U+\epsilon U\cdot\nablag U+LU=0$, where $U=(\zeta,\Vu)$, $B$ is symmetric and where $L$ is an anti-symmetric operator;
\item the operator $I+\mu A$ should be elliptic and of order at least two; therefore, one would have $$\nablag\zeta = -(I+\mu A)^{-1} ((I+\mu\Tb+\mu A)\dt \Vu + \epsilon R)$$ and a good control for $\dt\Vu$ would provide a good control for $\nablag\zeta$.\end{itemize}
 The two first constraints ensure the local existence for the new equation, while the third one ensures the large existence time. A short study shows that one should consider the following operator for $A$:
$$A\nablag W = -\nablag(\hbi\nablag\cdot(h_b W))$$ and the following symmetrizer for the equation:
\begin{equation}\mathcal{S} = \begin{pmatrix}
I-\mu\nablag\cdot(h_b\nablag\cdot) &0\\
0 &h_b
\end{pmatrix},\label{c3:symmetrizerS}\end{equation}
with an adapted change of unknown (see later) inspired by \cite{bresch_metivier}. However, with the consideration $\dt\Vu=-\nablag\zeta+O(\mu)$, one can for free make the operator $h_b(I+\mu\Tb+\mu A)$ elliptic by addition of the operator $\mu\nablago\nablago\cdot$, which provides a total control of a full derivative:
$$(h_b(I+\mu\Tb+A)\Vu,\Vu)_2 \sim \vert \Vu\vert_2^2+\mu\vert\nablag\cdot\Vu\vert_2^2+\mu\vert\nablago\cdot\Vu\vert_2^2,$$ where $$\nablago\cdot \Vu = (-\gamma\partial \Vu_x+\partial_x\Vu_y)$$ if $d=2$, and $\nablago\cdot = 0$ if $d=1$.  Remark that $\nablago\cdot\nablag\zeta = 0$, and since $d=1,2$, the operator $\nablago\cdot$ acts like the $\rm curl$ operator in dimension $3$.
\begin{remark}The operator $I+\mu\A=I-\mu \nablag(\hbi\nablag\cdot(h_b\cdot))$ is not elliptic, but it is invertible and its inverse gives precise control of the $H^1$ norm of the divergence, which is enough to control $\nabla\zeta$ in $H^1$ norm, since its curl is zero:$$\nablag\zeta=-(h_b(I+\mu \A))^{-1}(h_b(I+\mu\Tb+\mu\A-\nablag\nablago\cdot)\partial_t\Vu + \epsilon R).$$
\end{remark}

 We are therefore led to consider the following equation:
\begin{equation}
\left\{
\begin{aligned}
&\dt\zeta+\nablag\cdot(h\Vu)=0\\
&\big[I+\mu(\Tb-\nablag(\hbi\nablag\cdot(h_b\cdot))-\hbi\nablago\nablago\cdot)\big]\dt\Vu+\epsilon\Vu\cdot\nablag\Vu+(I-\mu\nablag\hbi\nablag\cdot(h_b\cdot))\nablag\zeta=0
\end{aligned}\right. \label{c3:boussinesq_modified_faux}
\end{equation}
However, the symmetrizer $\mathcal{S}$ defined by \eqref{c3:symmetrizerS} does not symmetrize properly the equation \eqref{c3:boussinesq_modified_faux}: there is a residual term in the first equation $\nablag\cdot(\epsilon\zeta\Vu)$ which is not canceled in the time derivative of the energy $(\mathcal{S}U,U)_2$ (with $U=(\zeta,\Vu)$). To overcome this problem, we use the following change of variable inspired by Brech-Métivier:
\begin{equation}
q=\frac{1}{\epsilon}\log(1+\frac{\epsilon\zeta}{h_b}). \label{c3:defq}
\end{equation}
The following Proposition is the key point of this change of variable, and states a precise relation between $q$ and $\zeta$:

\begin{proposition}Let $N>d/2+1$. Let also $h_b\in H^N(\R^d)$ be such that there exists $h_{\min}>0$ such that: $$\forall X\in\R^d,\qquad h_b(X) \geq h_{\min}.$$ Then, for $\epsilon$ small enough, the quantity $q$ defined by \eqref{c3:defq} is well defined. Moreover, one has: $$q=Q(\zeta)\zeta,$$ with $Q(\zeta)>0$. More precisely, for all $\alpha\in\mathbb{N}^d $,$\quad 1\leq \vert\alpha\vert\leq N$, one has $$\partial^\alpha q = Q_1(\zeta,h_b)\zeta^\alpha+P_\alpha(\zeta,h_b)$$ with $$Q_1(\zeta,h_b) = \int_0^t \frac{h_b}{h_b+\epsilon t\zeta }dt$$ and $$P_\alpha(\zeta,h_b) = \epsilon	 \sum_{\beta>0} Q_{\beta}(\zeta,h_b)+\sum_{0<\beta\leq\alpha} R_\beta(\zeta,h_b)\partial^\beta h_b$$ with $Q_\beta,R_\beta$ smooth functions of their arguments, and $$\vert P_\alpha(\zeta,h_b)\vert_{H^1} \leq C(h_{\min}, \vert h_b\vert_{H^N})\vert \partial^\alpha\zeta\vert_{2},$$ where $C$ is a smooth non decreasing function of its arguments.
\label{c3:qzeta}\end{proposition}
\begin{remark} The Proposition \ref{c3:qzeta} states that at the leading order, a derivative of $q$ is equal to a derivative of $\zeta$ up to a positive factor. Moreover, if one differentiates $q$ only with respect to time, since $\partial_t h_b = 0$ this equality is true up to an $\epsilon$ factor.\label{c3:remarkq}
\end{remark}
\begin{proof}
 Just notice that the definition of $q$ \eqref{c3:defq} implies \begin{equation}q=Q(\zeta)\zeta,\label{c3:qQ}\end{equation} with $$Q(\zeta) = \int_0^1 \frac{1}{h_b+\epsilon t\zeta}dt \zeta.$$ For a given $\zeta(0)$, the quantity $q$ is well-defined if there exists $h_{\min}>0$ such that the following condition is satisfied: $$\forall X\in\R^d,\qquad h_0(X) > h_{\min}.$$ Indeed,  for $\epsilon$ small enough one has  $h_b+\epsilon t\zeta (0)>0$ and this condition stay satisfied for $t$ small enough. Moreover, one has  $Q(\zeta)>0$. \par\vspace{\baselineskip} We now differentiate one time \eqref{c3:qQ} (the notation $\partial$ stands for any derivative of order one):
\begin{equation}
\partial q = \int_0^t \frac{h_b}{h_b+t\epsilon\zeta}dt \partial q
\label{c3:inter6}
\end{equation} which gives the expression of $Q_1(\zeta,h_b)$ given by the Proposition. The end of the proof is done by differentiating \eqref{c3:inter6}.
 
   \qquad$\Box$\end{proof}

\subsection{Local existence for the modified Boussinesq-Peregrine equation}\label{c3:section02}
We prove in this Section a local existence result for the modified Boussinesq-Peregrine equation \eqref{c3:boussinesq_modified_faux} introduced previously. We recall that we consider the change of unknown \begin{equation}q=\frac{1}{\epsilon}\log(1+\frac{\epsilon\zeta}{h_b})\label{c3:defq1}.\end{equation}
 Under the change of unknown \eqref{c3:defq1}, the equation \eqref{c3:boussinesq_modified_faux} takes the form:
\begin{equation}\left\{
\begin{aligned}
h_b\big(\dt Q(\zeta)\zeta+\epsilon\Vu\cdot\nablag(Q(\zeta)\zeta)\big)+\nablag\cdot(h_b\Vu)=0\\
h_b(I+\mu(\Tb-\nablag(\hbi\nablag\cdot(h_b\cdot))-\hbi\nablago\nablago\cdot))\dt\Vu&+h_b\epsilon\Vu\cdot\nablag\Vu\\&+h_b(I-\mu\nablag\hbi\nablag\cdot(h_b\cdot))\nablag\zeta=0.
\end{aligned}\right.\label{c3:modified_boussinesq}
\end{equation}
For the sake of clarity, we will use the following notations:
\begin{equation}
\Bi = (I+\mu\Tb-\mu\nablag(\hbi\nablag\cdot(h_b\cdot))-\mu\hbi\nablago\nablago\cdot),\quad \A = (I-\mu\nablag\hbi\nablag\cdot(h_b\cdot)),\qquad q= Q(\zeta)\zeta.\label{c3:defoperateurz}
\end{equation}
For all $N\in\mathbb{N}$, we define the following space:
\begin{equation}\mathcal{E}^N=\lbrace (\Vu,\zeta)\in H^N(\R^d)^d\times H^N(\R^d)\mid E^N(\Vu,\zeta)<\infty\rbrace\label{c3:energixspace}\end{equation} where
\begin{equation}
E^N(\zeta,\Vu) = \vert \zeta\vert_{H^N}+\sqrt{\mu}\vert\nablag\zeta\vert_{H^N}+\vert \Vu\vert_{H^N}+\sqrt{\mu}\vert \nablag \Vu\vert_{H^N}.\label{c3:energix}
\end{equation}
We denoted $\nablag\Vu$ the differential of $\Vu$. The space $\mathcal{E}^N$ endowed with the norm $E^N$ is a Banach space. We prove in this Section the following local existence result:
\begin{theorem}
Let $N\in\mathbb{N}$ be such that $N>d/2+2$. Let $h_b\in H^{N+1}(\R^d)$ be such that there exists $h_{\min}>0$ such that $$\forall X\in\R^d, h_b(X) \geq h_{\min}.$$ Let $(\zeta_0,\Vu_0)\in \mathcal{E}^N$. Then, there exists $T^*>0$ and a unique solution $(\Vu,\zeta)$ in $C([0;T^*[;\mathcal{E}^N)$ to the equation \eqref{c3:modified_boussinesq}. Moreover, one has:
\begin{equation}
\forall T<T^*,\qquad \forall \lambda \geq \underset{t\in[0;T]}{\sup} C(\hbi,\vert h_b\vert_{H^{N+1}},\vert \zeta,\Vu\vert_\Wdf),\qquad  E^N(t) \leq E^N(0) e^{\lambda t}. \label{c3:estimation_energie_modified}
\end{equation} \label{c3:existence_modified_boussinesq}
\end{theorem}
Even if at first sight the proof seems to follow the lines of the proof of local existence for a standard quasilinear hyperbolic system, as done for the Boussinesq-Peregrine equation in Section \ref{c3:section01}, we give here a detailed proof. Indeed, the energy \eqref{c3:energix} is defined in terms of $\zeta,\Vu$, while the equation \eqref{c3:modified_boussinesq} is expressed in terms of unknowns $\zeta, q,\Vu$, and the dependence of $q$ with respect to $\zeta$ is not trivial. This leads to technical complications that must be handled carefully.
\begin{remark}
In Theorem \ref{c3:boussi_theorem} which states the local existence for the initial Boussinesq-Peregrine equation, we used fractional order Sobolev spaces to define the energy of solutions, while we use here integer order Sobolev spaces. The reason is to have a coherent notation with the long time existence Theorem of Section \ref{c3:section03}, which can only be proved with an integer number of space derivatives, due to the method used. 
\end{remark}
\begin{proof}
As usual, the local existence follows the steps used for the quasilinear hyperbolic systems:\begin{itemize}[label=--,itemsep=0pt]
\item Step 1 : We solve a smoothed equation involving a mollifier $(1-\delta\Delta)$.
\item Step 2 : We prove that the existence time of the solution of the mollified equation does not depend on $\delta$, and the uniform bound in $H^s$ norm of this solution.
\item Step 3 : We pass to the limit $\delta$ goes to zero in the mollified equation to get a solution of the equation \eqref{c3:modified_boussinesq}.
\item Step 4 : We recover regularity for the solution.\end{itemize}

\textbf{Step 1}\qquad  One sets $\delta\in\R$ be such that $0<\delta<1$, and considers the following equation:
\begin{equation}
\left\{
\begin{aligned}
&(I-\delta\Delta)^2 \dt q^\delta+\epsilon\Vu\cdot\nablag q^\delta+\hbi\nablag\cdot(h_b\Vu)=0\\
&(I-\delta\Delta)h_b\Bi (I-\delta\Delta)\dt\Vu+\epsilon h_b\Vu\cdot\nablag\Vu+h_b\A\nablag\zetade=0.\label{c3:mollified_boussinesq}
\end{aligned}\right.\end{equation}
We first solve \eqref{c3:mollified_boussinesq} in the Banach space $\mathcal{E}^N$ defined by \eqref{c3:energixspace}. To this purpose, we note that the linear applications $(I+\delta\Delta)^{-1}$ and  $(I+\delta\Delta)^{-1}(h_b\Bi)^{-1}(I+\delta\Delta)^{-1}$ are respectively continuous from $H^{N-1}$ to $H^N$, and from $H^{N-3}$ to $H^N$, using the following Proposition:
\begin{proposition}\label{c3:regu_bi}
Let $N\geq 0$,  $t_0>d/2$, and $\beta\leq 1$. Let $b\in H^{N+1}(\R^d)$ be such that there exists $h_{\min}>0$ such that $h_b=1-\beta b\geq h_{\min}$. The operator $$h_b\Bi : H^{N+2}(\R^d)^d\longrightarrow H^N(\R^d)^d$$ is one-to-one and onto. One has, for all $V\in H^1(\R^d)^d$:
$$C_1(h_{\min})\vert V\vert_{H^1}^2\leq (h_b\Bi V,V)_2  \leq C_2(\vert h_b\vert_{H^{t_0+1}})\vert V\vert_{H^1}^2,$$ where the $C_i$ are non decreasing functions of their arguments.
 Moreover, one has, if $b\in H^{1+N\vee t_0}(\R^d)$ and for all $ f\in H^{N}(\R^d)^d$:
$$ \vert (h_b\Bi)^{-1} f\vert_{H^{N}}+\sqrt{\mu} \vert (h_b\Bi)^{-1}f\vert_{H^{N+1}} +\mu\vert(h_b\Bi)^{-1}f\vert_{H^{N+2}} \leq C(\hbi,\vert h_b\vert_{H^{1+N\vee t_0}})\vert f\vert_{H^N}$$ where $C$ is a non decreasing continuous function of its arguments.
\end{proposition}
The proof is an easy adaptation of the proof of the invertibility of the operator $I+\mu\Tb$ stated by Proposition \ref{c3:regu_inverse} (see \cite{david} Chapter 5 for a full proof). Just note that if $W\in L^2(\R^d)^d$, $(I+\mu\Tb)^{-1}W$ is only controlled with its divergence in $L^2(\R^d)$ norm, while $(h_b\Bi)^{-1}W$ is controlled in a full $H^1$ norm, due to the presence of the orthogonal gradient $\nablago$ in the operator $\Bi$ (see definition \eqref{c3:defoperateurz}). 

 Therefore, using Cauchy-Lipschitz theorem, there exists $T^\delta >0$ and a unique solution $(\zetadede,\Vde)\in C([0;T^\delta[,E^N)$  to the equation \eqref{c3:mollified_boussinesq} (just replace $\zetade $ by $ \frac{h_b}{\epsilon}(e^{\epsilon q^\delta}-1)$ in the second equation, to have an ODE in terms of the unknowns $(q^\delta,\Vde)$).  Moreover, $T^\delta <\infty$ if and only if \begin{equation*}
\underset{t\rightarrow T^\delta}{\lim} \vert (\zetadede,\Vde)(t)\vert_{E^N} = +\infty.
\end{equation*}

\textbf{Step 2}\qquad We now want to bound uniformly with respect to $\delta$ the energy $E^N$ defined by \eqref{c3:energix} of the unknowns.   We use the following notation: for all $\alpha\in\mathbb{N}^d, \vert\alpha\vert \leq N$, for all distribution $f$,
\begin{equation}
f_{(\alpha)}  = \partial^\alpha f. \label{c3:deffa}
\end{equation}
One differentiates the equation \eqref{c3:mollified_boussinesq} to find the following equation in terms of the unknowns $\zetadea,\Vda$ (recall the notation $\eqref{c3:deffa}$):
\begin{equation}
\left\{\begin{aligned}
&(I-\delta\Delta)^2\dt\qa +\epsilon\Vde\cdot\nablag\qa+\hbi\nablag\cdot(h_b\Vda)=R_1^\alpha \\
&(I-\delta\Delta)h_b\Bi(I-\delta\Delta)\dt\Vda + \epsilon h_b\Vde\cdot\nablag\Vda+h_b\A\nablag\zetadea=R_2^\alpha
\end{aligned}\right.\label{c3:quasilinear_boussi}
\end{equation}
where
\begin{equation*}
R_1^\alpha = -\epsilon\sum_{0<\beta\leq \alpha} \partial^\beta \Vde\cdot\nablag\partial^{\alpha-\beta}q^\delta -\sum_{0<\beta+\nu\leq \alpha}\partial^\beta( \hbi)\nablag\cdot(\partial^\nu (h_b)\partial^{\alpha-\beta-\nu}\Vde)
\end{equation*}
and
\begin{equation*}
R_2^\alpha = (I-\delta\Delta)[h_b\Bi,\partial^\alpha](I-\delta\Delta)\dt\Vde+\epsilon[h_b\Vde\cdot\nablag,\partial^\alpha]\Vde+[h_b\A,\partial^\alpha]\nablag\zetade.
\end{equation*}
As explained at the beginning of this Section, the system \eqref{c3:quasilinear_boussi} can be made symmetric by multiplying it by the following operator:
\begin{equation*}\begin{pmatrix}
h_b-\mu\nablag(h_b\nablag) &0\\
	0 &I\end{pmatrix}.
\end{equation*}
Note that according to Proposition \ref{c3:regu_bi}, one has $$\vert\zetade\vert_{H^N}^2+\mu\vert\nablag\zetade\vert_{H^N}^2+ (h_b\Bi V,V)_{H^N}\sim E^N(\zetade,\Vu)$$ and thus it is equivalent to control the $E^N$ norm of the unknown and the quantity $\vert\zetade\vert_{H^N}^2+\mu\vert\nablag\zetade\vert_{H^N}^2+ (h_b\Bi V,V)_{H^N}$. Following these considerations, one takes the $L^2$ scalar product of the first equation of \eqref{c3:quasilinear_boussi} with $(h_b-\mu\nablag\cdot(h_b\nablag\cdot))\zetadea$ and the scalar product of the second equation with $\Vda$. We obtain the following equality:
\begin{equation}
(T)+(V)+(Z) = (R^1_\alpha,h_b\zetadea-\mu\nablag\cdot(h_b\nablag\zetadea))_2+(R^2_\alpha,\Vda)_2, \label{c3:TVZ}
\end{equation}
where the time derivatives are 
\begin{equation}
(T) = (h_b\Bi(I-\delta\Delta)\dt \Vda,(I-\delta\Delta)\Vda)_2+((I-\delta\Delta)\dt\qa,h_b\zetadea-\mu\nablag\cdot(h_b\nablag\zetadea))_2, \label{c3:time_terms}
\end{equation}
the vanishing terms are 
\begin{equation}
(V)  = (\hbi\nablag\cdot(h_b\Vda),h_b\zetadea-\mu\nablag\cdot(h_b\nablag\zetadea))_2+(h_b\A\nablag\zetadea,\Vda)_2,\label{c3:vanishing_terms}
\end{equation}
and the terms of order zero to the contribution of the energy estimate:
\begin{equation}
(Z) = \epsilon(h_b\Vde\cdot\nablag\Vda,\Vda)_2+\epsilon(\Vd\cdot\nablag\qa,h_b\zetadea-\mu\nablag\cdot(h_b\nablag\zetadea))_2. \label{c3:symmetric_terms}
\end{equation}

\textit{- Control of the vanishing terms $(V)$}\qquad 
All has been made to conserve a certain symmetry in the equation, which is crucial here. Using the definition of $\A$ in the expression \eqref{c3:vanishing_terms}, one has \begin{align*}(V) &= (\nablag\cdot(h_b\Vda),\zetadea)_2+(h_b\nablag\zetadea,\Vda)_2  \\&-\mu(\hbi\nablag\cdot(h_b\Vda),\nablag\cdot(h_b\nablag\zetadea))_2-\mu(\nablag(\hbi\nablag\cdot(h_b\nablag\zetadea))_2,h_b\Vda)_2.\end{align*} By integrating by parts, the first two terms cancel one another, and the last two terms cancel one another. Therefore, $(V)$ actually vanishes.\par\vspace{\baselineskip}

\textit{- Control of the terms of order zero $(Z)$}\qquad
We start to control the easiest term of \eqref{c3:symmetric_terms}, which is the first one, by a classical symmetry trick:
\begin{align*}
\epsilon (h_b\Vde\cdot\nablag\Vda,\Vda)_2 &= \epsilon\sum_{j=1}^d (\Vde_j\partial_j\Vda,\Vda)_2 \\
&= -\epsilon \sum_{j=1}^d (\Vda,(\partial_j\Vd_j)\Vda)_2-\epsilon \sum_{j=1}^d (\Vda,\Vde_j\partial_j\Vda)_2
\end{align*}
by integrating by parts, and therefore 
$$\epsilon (h_b\Vde\cdot\nablag\Vda,\Vda)_2 = -\frac{1}{2}\epsilon \sum_{j=1}^d (\Vda,(\partial_j\Vde_j)\Vda)_2.$$ Using Cauchy-Schwarz inequality, one gets:
\begin{align}
\vert \epsilon (h_b\Vde\cdot\nablag\Vda,\Vda)_2 \vert \leq \epsilon\frac{1}{2}\vert \Vde\vert_{W^{1,\infty}}\vert\Vda\vert_2^2.\label{c3:symmetric1}
\end{align}
Note that the symmetry of this term is crucial here. For the other terms of $(Z)$ given in \eqref{c3:symmetric_terms}, the symmetry is less clear, since $\qa$ is not exactly $\zetadea$. We use the Proposition \ref{c3:qzeta} to compute:
\begin{align}
\epsilon(\Vde\cdot\nablag\qa,h_b\zetadea)_2 &= \epsilon(\Vde\cdot\nablag(Q_1(\zetade,h_b)\zetadea),h_b\zetadea)_2+\epsilon^2(\Vde\cdot\nablag( P_\alpha(\zetadede,h_b)),h_b\zetadea)_2. \label{c3:symmetric2calcul}
\end{align}
The second term of the right hand side of \eqref{c3:symmetric2calcul} is bounded using Cauchy-Schwarz inequality and Proposition \ref{c3:qzeta} by $$\epsilon^2C(\hbi,\vert h_b\vert_{H^N})\vert \zetadea\vert_2^2\vert \Vde\vert_{\infty}\vert h_b\vert_\infty,$$ where $C$ is a smooth non decreasing function of its arguments. The first term of the right hand side of \eqref{c3:symmetric2calcul} is bounded using the same symmetry trick as for the first term of $(V)$:
\begin{align*}
\epsilon(\Vde\cdot\nablag(Q_1(\zetade,h_b)\zetadea),h_b\zetadea)_2 &= \epsilon\sum_{j=1}^d (\Vde_j\partial_j (Q_1(\zetade,h_b)\zetadea),h_b\zetadea)_2\\
&= \epsilon\sum_{j=1}^d (\Vde_j(\partial_jQ_1(\zetade,h_b))\zetadea,h_b\zetadea)_2\\&-\epsilon\sum_{j=1}^d (\zetadea,\partial_j(Q_1(\zetade,h_b)\zetadea)\Vde_jh_b)_2\\&-\epsilon\sum_{j=1}^d (\zetadea,\partial_j(h_b\Vde_j)Q_1(\zetade,h_b)\zetadea)_2
\end{align*}
by integrating by parts. Therefore, using Cauchy-Schwarz's inequality, we get the bound:
$$\vert\epsilon(\Vde\cdot\nablag(Q_1\zetadea),h_b\zetadea)_2\vert \leq \frac{\epsilon}{2}\vert Q_1(\zetade,h_b)\vert_{W^{1,\infty}}\vert h_b\vert_{W^{1,\infty}}\vert\Vde\vert_{W^{1,\infty}}\vert\zetadea\vert_2^2$$
and finally, using the definition of $Q_1$ given by \eqref{c3:inter6}: 
\begin{equation}
\epsilon(\Vde\cdot\nablag\qa,h_b\zetadea)_2 \leq \epsilon C(\hbi,\vert \Vde\vert_\Wuf,\vert\zetadede\vert_\Wuf,\vert h_b\vert_{H^N})\vert\zetadea\vert_2^2	.\label{c3:symmetric2}
\end{equation}
The third term of the right hand side of the symmetric term \eqref{c3:symmetric_terms} is controlled with a similar technique: 
\begin{align}
\mu(\nablag(\Vde\cdot\nablag\qa),h_b\nablag\zetadea)_2 &=\mu(\nablag(\Vde\cdot\nablag(Q_1(\zetade,h_b)\zetadea)),h_b\nablag\zetadea)_2\nonumber\\&+\mu(\nablag(\Vde\cdot\nablag(\epsilon P_\alpha(\zetade,h_b))),h_b\nablag\zetadea)_2.\label{c3:symmetric3calcul}
\end{align}
We recall the identity:
$$\nablag (A\cdot B) = (A\cdot\nablag) B+(B\cdot\nablag) A + B\nablago\cdot A +A\nablago\cdot B.$$
 The first term of the right hand side of \eqref{c3:symmetric3calcul} can be expanded using this last identity:
\begin{equation}\begin{aligned}
\mu(\nablag(\Vde\cdot\nablag\qa),h_b\nablag\zetadea)_2 &=\mu(\Vde\cdot\nablag\nablag(Q_1(\zetade,h_b)\zetadea),h_b\nablag\zetadea)_2\\&+\mu(\Vde\nablago\cdot\nablag(Q_1(\zetade,h_b)\zetadea),h_b\nablag\zetadea)_2\\
&+\mu(\nablag(Q_1(\zetade,h_b)\zetadea)\cdot\nablag\Vde,h_b\nablag\zetadea)_2.
\end{aligned}\label{c3:inter7}\end{equation}
The first term of \eqref{c3:inter7} is a symmetric term, controlled by the same technique as before. The second term vanishes, and the last one with the last term  of the right hand side of \eqref{c3:symmetric3calcul} are easily controlled, and one gets:
\begin{equation}
\vert \mu(\nablag(\Vde\cdot\nablag\qa),h_b\nablag\zetadea)_2\vert \leq \mu C(\vert h_b\vert_{H^{N+1}},\vert\Vde\vert_\Wdf,\vert\zetadede\vert_\Wdf,\hbi)(\vert \nablag\zetadea\vert_2+\vert \zetadea\vert_2)\vert \nablag\zetadea\vert_2 .\label{c3:symmetric3}
\end{equation}
To conclude, putting together \eqref{c3:symmetric1},\eqref{c3:symmetric2} and \eqref{c3:symmetric3}, we proved that \begin{equation}
\vert(Z)\vert\leq \epsilon C(\vert h_b\vert_{H^{N+1}},\vert\Vde\vert_\Wdf,\vert\zetadede\vert_\Wdf,\hbi)E^N(\zetadede,\Vde).  \label{c3:symmetriccontrol}
\end{equation}

\textit{- Control of the time derivatives $(T)$}\qquad

The terms of $(T)$ involve time derivatives, and should be, up to terms controlled by the energy $E^N$, the time derivatives of the energy $E^N$. The first term of \eqref{c3:time_terms} is already symmetric, using the symmetry of $h_b\Bi$ (which is crucial here):
\begin{equation}(h_b\Bi(I-\delta\Delta)\dt\Vda,(I-\delta\Delta)\Vda)_2 = \dt \frac{1}{2}(h_b\Bi(I-\delta\Delta)\Vda,(I-\delta\Delta)\Vda)_2.\label{c3:T1}\end{equation}
For the second term of \eqref{c3:time_terms}, we use again Proposition \ref{c3:qzeta} to write:
\begin{align}
((I-\delta\Delta)^2\dt\qa,h_b\zetadea)_2 &=((I-\delta\Delta)^2\dt (Q_1\zetadea),h_b\zetadea)_2+((I-\delta\Delta)^2\dt( \epsilon P_\alpha\zetadede),h_b\zetadea)_2.\label{c3:time_calcul1}
\end{align}
For the first term of the right hand side of \eqref{c3:time_calcul1}, one computes:
\begin{equation}\begin{aligned}
((I-\delta\Delta)^2\dt (Q_1\zetadea),h_b\zetadea)_2 &= ((I-\delta\Delta)^2\dt (Q_1)\zetadea,h_b\zetadea)_2+((I-\delta\Delta)^2 Q_1\dt\zetadea,h_b\zetadea)_2 \\
&= ((I-\delta\Delta)^2\dt (Q_1)\zetadea,h_b\zetadea)_2\\ &+ ([Q_1,-\delta\Delta]\dt\zetadea,(I-\delta\Delta)h_b\zetadea)_2\\
&+(Q_1(I-\delta\Delta)\dt\zetadea,[-\delta\Delta,h_b]\zetadea)_2\\&+\dt \frac{1}{2}(Q_1(I-\delta\Delta)\zetadea,h_b(I-\delta\Delta)\zetadea)_2.
\end{aligned}\label{c3:inter8}\end{equation}
All these computations are made to obtain the time derivative of a symmetric term with respect to $\zetadea$. The first term of the right hand side of \eqref{c3:inter8} is easily controlled by
$$\vert((I-\delta\Delta)^2\dt (Q_1)\zetadea,h_b\zetadea)_2 \vert \leq \vert \dt Q_1\vert_\Wuf\vert (I-\delta\Delta)\zetadea\vert_2^2 \vert h_b\vert_\Wdf.$$
In order to control the second term of the right hand side of \eqref{c3:inter8}, we  replace $\dt\zetadea$ by its expression given by the equation:
$$\dt\zetadea = -\partial^\alpha(\nablag\cdot(h^\delta\Vde))$$ and we notice that $$\vert\delta(I-\delta\Delta)^{-2}[\Delta,Q_1] \partial^\alpha \nablag\cdot u \vert_2 \leq C(\vert Q_1\vert_\Wdf) \vert\partial^\alpha u\vert_2$$ for all $u$ in $H^N$. Therefore, one gets:
\begin{equation}
((I-\delta\Delta)^2\dt (Q_1\zetadea),h_b\zetadea)_2 = \dt \frac{1}{2}(Q_1(I-\delta\Delta)\zetadea,h_b(I-\delta\Delta)\zetadea)_2 +R\label{c3:T2}
\end{equation}
with $$\vert R\vert \leq C(\vert h_b\vert_{H^N},\hbi,\vert \zetadede\vert_\Wdf,\vert\Vde\vert_\Wdf)(\vert(I-\delta\Delta)\zetadea\vert_2 + \vert \Vda\vert_2)\vert(I-\delta\Delta)\zetadea\vert_2.$$
The same technique can be used for the control of the second term of \eqref{c3:time_calcul1} and one gets finally, combining \eqref{c3:T1} and \eqref{c3:T2}:
\begin{equation}
(T) = \dt \frac{1}{2}(h_b\Bi(I-\delta\Delta)\Vda,(I-\delta\Delta)\Vda)_2 + \dt \frac{1}{2}(Q_1(I-\delta\Delta)\zetadea,h_b(I-\delta\Delta)\zetadea)_2 +R \label{c3:timecontrol1}\end{equation}with
\begin{equation}
 \vert R\vert \leq \mu C(\vert h_b\vert_{H^N},\hbi,\vert \zetadede\vert_\Wdf,\vert\Vde\vert_\Wdf)E^N((I-\delta\Delta)\zetadede,(I-\delta\Delta)\Vde).\label{c3:timecontrol2}\end{equation}

\textit{- Control of the residual terms}\qquad
We now control the terms involving the residuals that appear in \eqref{c3:TVZ}. One has: 
\begin{equation*}
R_1^\alpha = -\epsilon\sum_{0<\beta\leq\alpha} \partial^\beta \Vde\cdot\nablag\partial^{\alpha-\beta}q^\delta -\sum_{0<\beta+\nu\leq\alpha}\partial^\beta( \hbi)\nablag\cdot(\partial^\nu (h_b)\partial^{\alpha-\beta-\nu}\Vde)
\end{equation*}
and thus one has, using a Kato-Ponce type estimate (of the form of Proposition \ref{c3:katoponce}):
\begin{equation}\begin{aligned}\vert R_1^\alpha\vert_2+\sqrt{\mu}\vert \nablag R_1^\alpha\vert_2 \leq C(\hbi,\vert h_b\vert_{H^{N+1}})\vert(\zetadede,\Vde)\vert_{W^{2,\infty}}\times \\(\vert \zetadea\vert_2+\sqrt{\mu}\vert \nablag\zetadea\vert_2+\vert\Vda\vert_2+\sqrt{\mu}\vert \nablag \Vda\vert_2).\end{aligned}\label{c3:reste1}\end{equation}
It is very important to have $\nablag\cdot(h_b\Vda)$ instead of $\nablag\cdot(h\Vda)$ in the equation \eqref{c3:quasilinear_boussi}, because the term $\nablag\cdot(\epsilon\zetadea\Vda)$  would not be  properly symmetrized and thus would not be controlled by the energy.
One has easily, integrating by parts and using Cauchy-Schwarz inequality:
\begin{align}\vert(R_1^\alpha,h_b\zetadea-\mu\nablag\cdot(h_b\nablag\zetadea))_2\vert \leq \vert R_1^\alpha\vert_2\vert h_b\vert_\Wzf\vert\zetadea\vert_2 +\mu\vert\nablag R_1^\alpha\vert_2\vert h_b\vert_\Wzf\vert\nablag\zetadea\vert_2 \label{c3:rest1}
\end{align} 
and thus, using \eqref{c3:reste1}, one gets:
\begin{equation}
\vert(R_1^\alpha,h_b\zetadea-\mu\nablag\cdot(h_b\nablag\zetadea))_2\vert \leq C(\hbi,\vert h_b\vert_{H^{N+1}},\vert(\zetadede,\Vde)\vert_{\Wdf}) (E^N)^2.  \label{c3:reste2}
\end{equation}

 Recall that \begin{equation}
R_2^\alpha = (I-\delta\Delta)[h_b\Bi,\partial^\alpha](I-\delta\Delta)\dt\Vde+\epsilon[h_b\Vde\cdot\nablag,\partial^\alpha]\Vde+[h_b\A,\partial^\alpha]\nablag\zetadede \label{c3:rest_terms}
\end{equation}
To control the first term of \eqref{c3:rest_terms}, as usual one replaces $(I-\delta\Delta)\dt\Vde$ by its expression given by the equation \eqref{c3:quasilinear_boussi}, and uses the definition of $\Tb$ given by \eqref{c3:deftb}:
\begin{equation}\begin{aligned}
(I-\delta\Delta)[h_b\Bi,\partial^\alpha](I-\delta\Delta)\dt\Vde \\= -(I-\delta\Delta)[h_b\mu\Tb,\partial^\alpha]  (h_b\Bi)^{-1}(I-\delta\Delta)^{-1}\big( \epsilon h_b\Vde\cdot\nablag\Vde +h_b\A\nablag\zetadede	\big).
\end{aligned}\end{equation}\label{c3:truc5}
One has:
\begin{equation}\forall k\geq 2,\qquad \mu\vert (h_b\Bi)^{-1} u\vert_{H^k} \leq C (\hbi,\vert h_b\vert_{H^{N+1}})  \vert u \vert_{H^{k-2}}\label{c3:truc4}\end{equation}
using Proposition \ref{c3:regu_bi}, for $C$ a smooth non decreasing function of its arguments. One has also:
\begin{equation}\forall k\geq 0,\qquad \vert (I-\delta\Delta) u \vert_{H^k} \leq C (\vert u\vert_{H^k}+\delta \vert u\vert_{H^{k+2}})\label{c3:truc1}\end{equation}
with $C$ independent on $\delta$, and:
\begin{equation}\forall k\geq 0,\qquad\frac{1}{\delta}\vert (I-\delta\Delta)^{-1} u\vert_{H^{k+2}}+\vert (I- \delta\Delta)^{-1} u\vert_{H^k}\leq \vert u\vert_{H^k}.\label{c3:truc2}\end{equation}
Using the definition of $\Tb$ (see \eqref{c3:deftb}), one has:
\begin{equation}\forall k\geq 0,\qquad \mu\vert [h_b\Tb,\partial^\alpha] u\vert \leq \mu C(\vert h_b\vert_{H^{N+1}})\vert u\vert_{H^{k+\alpha+1}}.\label{c3:truc3}\end{equation}
Using successively the identities \eqref{c3:truc1}, \eqref{c3:truc3}, \eqref{c3:truc4} and \eqref{c3:truc2},   the first term  of the right hand side of \eqref{c3:truc5} is bounded by:
$$\vert(I-\delta\Delta)[h_b\mu\Tb,\partial^\alpha]  (h_b\Bi)^{-1}(I-\delta\Delta)^{-1}\big( \epsilon h_b\Vde\cdot\nablag\Vde)\vert_2\leq C (\hbi,\vert h_b\vert_{H^{N+1}}) \vert \Vda\vert_2.$$

One has to be more careful for the second term of the right hand side of \eqref{c3:truc5}, because the expression $\A\nablag\zetade$ is of order $2$ in $\mu\nablag\zetade$. One writes:
\begin{align*}\vert((I-\delta\Delta)[h_b\mu\Tb,\partial^\alpha]  (h_b\Bi)^{-1}(I-\delta\Delta)^{-1}( h_b \A\nablag\zetade),\Vda)_2\vert \leq\\
\sqrt{\mu}\vert((I-\delta\Delta)[h_b\mu\Tb,\partial^\alpha]  (h_b\Bi)^{-1}(I-\delta\Delta)^{-1}( h_b \A\nablag\zetade)\vert_{H^{-1}}\sqrt{\mu}\vert \Vda\vert_{H^1}\end{align*}
and we use the same controls \eqref{c3:truc1}, \eqref{c3:truc2}, \eqref{c3:truc4}, \eqref{c3:truc5}  as before. Finally, one gets:
\begin{align}
\vert (R_2^\alpha,\Vda)\vert \leq C (\hbi,\vert h_b\vert_{H^{N+1}})( \vert \Vda\vert_2^2+\mu\vert\nablag\Vda\vert_2\vert\nablag\zetadea\vert_2).\label{c3:rest2}
\end{align}
\textit{- Conclusion}

Putting together \eqref{c3:symmetriccontrol}, \eqref{c3:timecontrol1}, \eqref{c3:timecontrol2}, \eqref{c3:rest1} and \eqref{c3:rest2}, one gets:
\begin{align*}\dt \frac{1}{2}(h_b\Bi(I-\delta\Delta)\Vda,(I-\delta\Delta)\Vda)_2 + \dt \frac{1}{2}(Q_1(I-\delta\Delta)\zetadea,h_b(I-\delta\Delta)\zetadea)_2\\  \leq C(\hbi,\vert h_b\vert_{H^{N+1}},\vert \zetadede,\Vde\vert_\Wdf)\times 
E^N((I-\delta\Delta)\zetadede,(I-\delta\Delta)\Vde).\end{align*}
We recall that using Proposition \ref{c3:regu_bi}, one has$$(h_b\Bi V,V)_2\sim \vert V\vert_2^2+\mu\vert \nablag V\vert_2^2$$ and using Proposition \ref{c3:qzeta}, one has $\vert Q_1\vert_{\infty} \geq C(\hbi)$. Therefore, we obtained:
\begin{equation}
\dt E^N((I-\delta\Delta)\zetadede,(I-\delta\Delta)\Vde)\leq  C(\hbi,\vert h_b\vert_{H^{N+1}},\vert \zetadede,\Vde\vert_\Wdf)E^N((I-\delta\Delta)\zetadede,(I-\delta\Delta)\Vde) \label{c3:esimation_energie1}
\end{equation}
where $C$ is a non decreasing continuous function of its arguments, independent on $\delta$. Therefore, using Gronwall's Lemma, $T^\delta$ does not depends on $\delta$. 

\textbf{Step 3-4}\qquad  The rest of the proof is exactly the same as for the local existence of the standard Boussinesq-Peregrine equation, and one gets a solution to \eqref{c3:mollified_boussinesq} on a time interval $[0;T^*[$. One gets however from \eqref{c3:esimation_energie1} that:
\begin{equation*}
\forall T<T^*,\qquad \forall \lambda \geq \underset{t\in[0;T]}{\sup} C(\hbi,\vert h_b\vert_{H^{N+1}},\vert \zeta,\Vu\vert_\Wdf),\qquad  E^N(t) \leq E^N(0) e^{\lambda t}. 
\end{equation*} 
\end{proof}

\subsection{Long time existence in dimension $1$ for the modified Boussinesq-Peregrine equation}\label{c3:section03}
We now make the scaling $$t'=\epsilon t$$ on the equation \eqref{c3:modified_boussinesq}, and we obtain the equation (we get rid of the "primes" in the notation $t'$ for the sake of clarity):
\begin{equation}
\left\{
\begin{aligned}
& \dt q+\Vu\cdot\nablag q+\frac{1}{\epsilon}\hbi\nablag\cdot(h_b\Vu)=0\\
&h_b\Bi \dt\Vu+ h_b\Vu\cdot\nablag\Vu+\frac{1}{\epsilon} h_b\A\nablag\zeta=0.\label{c3:rescaled_bousisnesq}
\end{aligned}\right.\end{equation}
This change of variable is not necessary on a mathematical point of view, but it allows to highlight the singular terms that must be canceled in the energy estimates in order to prove the result, which are the large terms of size $\frac{1}{\epsilon}$. Moreover, it allows the equation \eqref{c3:rescaled_bousisnesq} to be seen as a singular perturbation problem. Note that a time existence of size $\frac{1}{\epsilon}$ for the equation \eqref{c3:modified_boussinesq} is equivalent to a time existence independent on $\epsilon$ for \eqref{c3:rescaled_bousisnesq}. \par\vspace{\baselineskip}
We recall that for all $N\in\mathbb{N}$, we define the following space:
\begin{equation*}\mathcal{E}^N=\lbrace (\Vu,\zeta)\in H^N(\R^d)^d\times H^N(\R^d)\mid E^N(\Vu,\zeta)<\infty\rbrace\end{equation*} where
\begin{equation*}
E^N(\zeta,\Vu) = \vert \zeta\vert_{H^N}+\sqrt{\mu}\vert\nablag\zeta\vert_{H^N}+\vert \Vu\vert_{H^N}+\sqrt{\mu}\vert \nablag \Vu\vert_{H^N}.
\end{equation*}
We prove in this section the following result:
\begin{theorem}
Let $d=1$. Let $N\in\mathbb{N}$ be such that $N>d/2+2$. Let $h_b\in H^{N+1}(\R^d)$ be such that there exists $h_{\min}>0$ such that $$\forall X\in\R^d, h_b(X) \geq h_{\min}.$$ Let $(\zeta_0,\Vu_0)\in \mathcal{E}^N$. Then, there exists $T>0$ and a unique solution $(\Vu,\zeta)$ in $C([0;T[;\mathcal{E}^N)$ to the equation \eqref{c3:rescaled_bousisnesq}, with $$T=C_1(E^N(\zeta_0,\Vu_0),\hbi,\vert h_b\vert_{H^{N+1}}),$$ where $C_1$ is a non decreasing continuous function of its arguments. \label{c3:large_time_theorem}
\end{theorem}
In particular, the time of existence does not depend on $\epsilon,\mu$. 
\begin{remark}It is very important to note that $d=1$ here. In $d=2$, there is an extra difficulty due to the need of a good estimate for $\nablago\cdot\Vu$. However, since this is the only difficulty that could prevent a similar result in dimension $d=2$ to hold, we keep the notations of the multidimensional equation, and we specifically highlight at the end of the proof the difficulty that one must overcome to prove the result in dimension $2$. 
\end{remark}
Let us consider $(\Vu,\zeta)$ the unique solution of \eqref{c3:rescaled_bousisnesq} given by Theorem \ref{c3:existence_modified_boussinesq} on a time interval $[0;T^\epsilon]$. We set $$K = \underset{t\in[0;T^\epsilon]}{\sup} E^N(\zeta,\Vu).$$ We use the notation $$u_k = (\epsilon\dt)^k u$$ for all distribution $u$ (thus $u_k$ corresponds to the time derivative of $u$ in the original time variables).\par\vspace{\baselineskip}
The idea of the proof is to obtain a "good" energy estimate of the form $$E(t)\leq C(K)(t+\epsilon)+C_0,$$ where $C$ is non decreasing and smooth, and where $C_0$ only depends on the initial data. Such estimate would allow us to get by a continuity argument a time existence uniform with respect to $\epsilon$. There are two main ideas in the proof:
\begin{itemize}[label=--,itemsep=0pt]
\item The system is still symmetric with respect to singular terms if we differentiate it with respect to time. It allows us to get the "good estimate" for the time derivatives $\zeta_k,\Vu_k$. 
\item Using the equation, one can control the space derivatives by the time derivatives, and recover the "good estimate" for the full energy $E^N$ of the solutions. 
\end{itemize}

The following Proposition states that the time derivatives of the solutions $(\Vu,\zeta)$ have the same regularity as the space derivatives:
\begin{proposition}
One has, for all $0\leq k\leq N$, $$\vert (\Vk,\zetk)\vert_{H^{N-k}}+\sqrt{\mu}\vert (\Vk,\zetk)\vert_{H^{N-k+1}} \leq C(K),$$ where $C$ is a smooth, non decreasing function of its argument.\label{c3:regu_temps}
\end{proposition}
\begin{proof}
For $k=0$, it is clear. Suppose it is true for $k\geq 0$. One commutes $(\epsilon\partial_t)^k$ with the equation \eqref{c3:rescaled_bousisnesq}. One gets, since $\dt h_b =0$:
\begin{align*}
\begin{cases}
q_{k+1}+\epsilon\sum_{j=0}^k \Vu_j\cdot\nablag q_{k-j}+\nablag\cdot(h_b \Vk) =0\\
\Vu_{k+1} =- (h_b\Bi)^{-1} (\epsilon h_b \sum_{j=0}^k \Vu_j\cdot \nablag\Vu_{k-j}+h_b\A \nablag\zetk).
\end{cases}
\end{align*}
We only prove the most difficult estimate which is the following, in order to prove that the induction hypothesis is true at rank $k+1$:
$$\vert(h_b\Bi)^{-1} (h_b\A \nablag\zetk )\vert_{H^{N-k-1}}\leq C(K).$$ We recall that $$\A = I-\mu\nablag\hbi\nablag\cdot(h_b\cdot)$$ and therefore, using Proposition \eqref{c3:regu_bi}:
\begin{align*}
\vert(h_b\Bi)^{-1} (h_b\A \nablag\zetk )\vert_{H^{N-k-1}} &\leq \vert(h_b\Bi)^{-1} (h_b \nablag\zetk )\vert_{H^{N-k-1}}+\mu \vert(h_b\Bi)^{-1}( \nablag\hbi\nablag\cdot(h_b\nablag\zetk ))\vert_{H^{N-k-1}}\\
&\leq C(\hbi,\vert h_b\vert_{H^{N+1}})\vert\zetk\vert_{H^{N-k}}
\end{align*}
and one gets the desired control by using the induction hypothesis. The other controls are done similarly, using Proposition \ref{c3:regu_bi}, and the relation between $q_k$ and $\zeta_k$ given by Proposition \ref{c3:qzeta} and Remark \ref{c3:remarkq}.\qquad$\Box$
\end{proof}
The key point of the proof of Theorem \ref{c3:large_time_theorem} is the following Lemma, which states a "good estimate" for the unknowns:
\begin{lemma} One has $$E^N(\zeta,\Vu)\leq C(K)(t+\epsilon)+C_0$$ where $C$ is a non decreasing function of its arguments, and $C_0$ is a constant which only depends on the initial data.\label{c3:keylemma}
\end{lemma}
\begin{proof}
There are two ideas in the proof of this lemma:\begin{itemize}[label=--,itemsep=0pt]
\item the time derivatives of the unknowns satisfy a system which is still symmetric with respect to singular terms of size $\frac{1}{\epsilon}$; 
\item  the space derivatives are related to time derivatives by the equation.\end{itemize}
\par The unknowns $(\zetk,\Vk)$ satisfy the following equation:
\begin{equation}
\left\{\begin{aligned}
&\dt q_k +\Vu\cdot\nablag q_k+\frac{1}{\epsilon}\hbi\nablag\cdot(h_b\Vk)=R^1_k \\
&h_b\Bi\dt\Vk + h_b\Vu\cdot\nablag\Vk+\frac{1}{\epsilon}h_b\A\nablag\zeta_k=R^2_k
\end{aligned}\right.\label{c3:quasilinear_boussi_temps}
\end{equation}
where \begin{equation}
R^1_k = [\Vu,(\epsilon\dt)^k] q,\qquad R^2_k = [\Vu\cdot\nablag,(\epsilon\dt)^k] \Vu .\label{c3:rest_terms_temps}
\end{equation}
The symmetry with respect to large terms of size $\frac{1}{\epsilon}$ is conserved, which allows to get the following result:
\begin{lemma} \label{c3:bonne_estimation_temps}
One has, for all $0\leq k\leq N$, $$E^0(\zetk,\Vk) \leq C(K)t+C_0.$$
\end{lemma}
\begin{proof}
The equation \eqref{c3:quasilinear_boussi_temps} is still symmetric with respect to large terms of size $\frac{1}{\epsilon}$. More precisely, if one multiplies the first equation by $h_b\zetk-\mu\nablag\cdot(h_b\nablag \zetk)$, one finds exactly as in the proof of Theorem \ref{c3:existence_modified_boussinesq} an expression of the form:
$(T)+(V)+(Z) = (R^1_k,h_b\zetk-\mu\nablag\cdot(h_b\nablag\zetk))_2+(R^2_k,\Vk)_2$ with exactly the same terms for $(T)$, $(V)$ and $(Z)$ as in \eqref{c3:time_terms}, \eqref{c3:symmetric_terms} with $(\zetaa,\Vda)$ replaced by $(\zetk,\Vk)$. The vanishing terms are exactly ones of size $\frac{1}{\epsilon}$ and the others are controlled exactly with the same techniques, using Proposition \ref{c3:regu_temps} for the regularity of the time derivatives.\qquad $\Box$
\end{proof}
Now, we recover the "good estimate"  of Lemma \ref{c3:keylemma} for the space derivatives of the unknowns, using the equation.
\begin{lemma}
One has, for all $0\leq k\leq N$, $$E^{N-k}(\zetk,\Vk)\leq C(K)(t+\epsilon)+C_0.$$
\end{lemma}
\begin{proof}
We prove it by backward finite induction on $k$. For $k=N$, it is Lemma \ref{c3:bonne_estimation_temps}. Suppose it is true for $k+1$ with $k\leq N-1$. Let us prove it is true for $k$.

\begin{equation}\nablag\zeta_k = -(h_b\A)^{-1} (h_b\Bi \Vu_{k+1}+\epsilon h_b\Vu\cdot\nablag\Vk-\epsilon R_k^2).\label{c3:recover_zeta}\end{equation}
Recall that the operator $\A$ is given by $$\A = (I-\mu\nablag\hbi\nablag\cdot(h_b\cdot)).$$ We also recall that we defined $\vert f\vert_{X^N}$ for $f\in L^2(\R^d)^d$ by: 
$$\vert f\vert_{X^N}^2 = \vert f\vert_{H^N}^2+\mu\vert \nablag\cdot f\vert_{H^N}^2.$$We used the following Proposition that states the invertibility of $h_b\A$ to derive the equality \eqref{c3:recover_zeta}:
\begin{proposition}\label{c3:reguA}
Let $N\in\mathbb{N}$ and let $h_b\in H^N(\R^d)$ be such that there exists $h_{\min}>0$ such that $$\forall X\in\R^d,\qquad h_b(X)\geq h_{\min}.$$ We set, for all $f\in L^2(\R^d)$ : $$\vert f\vert_{X^N}^2 = \vert f\vert_{H^N}^2+\mu\vert\nablag\cdot f\vert_{H^N}^2.$$  The operator $h_b\A$ is invertible on $H^N(\R^d)^d$. Moreover, the following estimates stand.\begin{enumerate}[label = ( \arabic*)]\item   For all $f\in H^N(\R^d)^d$,
$$ \vert (h_b\A)^{-1} f\vert_{X^N} \leq C(\hbi,\vert h_b\vert_{H^{N+1}})\vert f\vert_{H^N}.$$
\item For all $g\in H^N(\R^d)$,
$$ \sqrt{\mu}\vert (h_b\A)^{-1} \nablag g\vert_{X^N} \leq C(\hbi,\vert h_b\vert_{H^{N+1}})\vert g\vert_{H^N}.$$\end{enumerate}
\end{proposition}
We postpone the proof of Proposition \ref{c3:reguA} to Appendix \ref{c3:appendix_operator} for the sake of clarity. Now, in order to use the relation \eqref{c3:recover_zeta}, one takes the $H^{N-k-1}$ scalar product of \eqref{c3:recover_zeta} with $h_b(I-\mu\nablag\nablag\cdot)\nablag\zeta$, and gets, using the notations of Proposition \ref{c3:reguA} :
\begin{equation}
\vert\nablag\zeta_k\vert_{X^{N-k-1}}^2 = -(h_b(I-\mu\nablag\nablag\cdot)\nablag\zeta_k, (h_b\A)^{-1} (h_b\Bi \Vu_{k+1}+\epsilon h_b\Vu\cdot\nablag\Vk-\epsilon R_k^2))_{H^{N-k-1}}.
\label{c3:estimation_zetak}
\end{equation}

Now, one has by definition of $R^2_k$ given by \eqref{c3:rest_terms_temps} : $$h_b\Vu\cdot\nablag\Vu_k-R^2_k = (\epsilon_t)^k (\Vu\cdot\nablag\Vu)$$ and thus this term is sum of terms of the form $$\Vu_l\cdot\nablag\Vu_{k-l},$$ with $0\leq l\leq k$ and therefore one has, using Proposition \ref{c3:reguA}: \begin{align}
\vert  (h_b\A)^{-1} (h_b\Vu\cdot\nablag\Vk-R^2_k)\vert_{X^{N-k-1}} &\leq C(\frac{1}{h_{\min}} ,\vert h_b\vert_{H^{N+1}}) \vert  h_b\Vu\cdot\nablag\Vu_k-R^2_k\vert_{H^{N-k-1}}\nonumber \\
&\leq  C(K).\label{c3:control_transport_zeta}
\end{align}
We now focus on the control of  $((h_b\A)^{-1} (h_b\Bi)\Vu_{k+1},h_b(I-\mu\nablag\nablag\cdot)\nablag\zeta_k)_{H^{N-k-1}}$. Recall that $$\Bi = (I+\mu\Tb-\mu\nablag(\hbi\nablag\cdot(h_b\cdot))-\mu\hbi\nablago\nablago\cdot),\quad \A = (I-\mu\nablag\hbi\nablag\cdot(h_b\cdot)).$$

\begin{lemma}\label{c3:reguATB}
One has, for all $V,W\in H^{k+1}(\R^d)^d$, all $0\leq k\leq N$: $$((h_b\A)^{-1}h_b\Bi V,h_b\nablag W)_{H^k}\leq   C(\hbi,\vert h_b\vert_{H^{N+1}}) \sqrt{\mu}\vert \nablag V\vert_{H^{k}} \vert \nablag W\vert_{H^{k}}.$$
\end{lemma}
\begin{remark}
This Lemma states that even if $(h_b\A)$ is not elliptic (it is essentially $I-\mu\nabla\nablag\cdot$ with variables coefficients), its inverse allows to recover a full derivative if it is applied to a gradient. The quantity $h_b\Bi$ is essentially composed of gradients, except for the term $\nablago\nablago\cdot$ which vanishes in any scalar product with a gradient.
\end{remark}
\begin{proof}

We only give the control of the most difficult terms of the quantity to be controlled, which are: \begin{equation}
\mu ((h_b\A)^{-1}(\nablag(h_b^3\nablag\cdot V)), h_b\nablag W)_{H^{k}},\qquad  \mu ((h_b\A)^{-1}(\nablago\nablago\cdot V),h_b\nablag W)_{H^{k}}\label{c3:terme_difficiles}
\end{equation}
(the other terms from $h_b\Bi$ are controlled even more easily by similar techniques). For the first term of \eqref{c3:terme_difficiles}, one computes:
\begin{align*}
\mu ((h_b\A)^{-1}(\nablag(h_b^3\nablag\cdot V)), h_b\nablag W)_{H^{k}} &\leq  \mu \vert (h_b\A)^{-1}\nablag(h_b^3\nablag\cdot V)\vert_{H^{k}}\vert h_b\nablag W\vert_{H^{k}} \\
&\leq \sqrt{\mu} C(\hbi,\vert h_b\vert_{H^{N+1}})\vert h_b^3\nablag\cdot V\vert_{H^{k}}\vert\nablag W\vert_{H^{k}}
\end{align*}
where we used Proposition \eqref{c3:reguA} to derive the last inequality. Finally, one gets:
\begin{equation}
\mu ((h_b\A)^{-1}(\nablag(h_b^3\nablag\cdot V)), h_b\nablag W)_{H^{k}} \leq C(\hbi,\vert h_b\vert_{H^{N+1}}) \sqrt{\mu}\vert \nablag V\vert_{H^{k}} \vert \nablag W\vert_{H^{k}} .  \label{c3:control_tb_zeta}
\end{equation}

 For the second term of \eqref{c3:terme_difficiles}, one computes, integrating by parts:
\begin{equation}\begin{aligned}
\mu ((h_b\A)^{-1}(\nablago\nablago\cdot V),h_b\nablag W)_{H^{k}} \\ 
= -\mu (\nablag\cdot h_b(h_b\A)^{-1}(\nablago\nablago\cdot V),W)_{H^{k}}\\
+\mu ( \Lambda^{k-1} (h_b\A)^{-1}(\nablago\nablago\cdot V),\Lambda [h_b,\Lambda^{k}] \nablag W)_2\\
+\mu ( [h_b,\Lambda^{k}] (h_b\A)^{-1}(\nablago\nablago\cdot V),\Lambda^{k}\nablag W)_2,
\end{aligned}  \label{c3:nablagozeta} \end{equation} where we recall that $\Lambda = (1+\D^2)^{1/2}$. One has to notice that for $f\in H^1(\R^d)$, $u = (h_b\A)^{-1}(h_b\nablag f)$ is a term of the form $\nablag g$, since $u=\nablag(\hbi\nablag\cdot(h_b u))+\nablag f$, using the definition of $\A$ given by \eqref{c3:defoperateurz}. Therefore, $\nablago\cdot (h_b\A)^{-1}(h_b\nablag f)=0$ for all $f$, and by duality $\nablag\cdot (h_b(h_b\A)^{-1}\nablago w)=0$ for all $w$. The first term of the rhs of term \eqref{c3:nablagozeta} is therefore zero. \par\vspace{\baselineskip} For the second term of the rhs of \eqref{c3:nablagozeta}, one easily proves that, for all $f\in H^k(\R^d)$, using the Kato-Ponce estimate of Proposition \ref{c3:katoponce}: $$\vert \Lambda [h_b,\Lambda^{k}] f\vert_2 \leq C(\vert h_b\vert_{H^{N+1}}) \vert f\vert_{H^{k}},$$ and thus the second term of the rhs of \eqref{c3:nablagozeta}   is bounded by
$$\mu C(\vert h_b\vert_{H^{N+1}})	\vert (h_b\A)^{-1}(\nablago\nablago\cdot V)\vert_{H^{k-1}}\vert \nablag W\vert_{H^{k}}$$ and using Proposition \eqref{c3:reguA}, one gets the bound 
\begin{equation}
\mu \vert ( \Lambda^{k-1} (h_b\A)^{-1}(\nablago\nablago\cdot V),\Lambda [h_b,\Lambda^{k}]\nablag \nablag W)_2 \vert \leq \mu C(\hbi,\vert h_b\vert_{H^{N+1}}) \vert \nablag V \vert_{H^{k}}\vert \nablag W\vert_{H^{k}}.\label{c3:bound_zeta}
\end{equation}
The third term of \eqref{c3:nablagozeta} is controlled similarly with the same bound as \eqref{c3:bound_zeta}. Putting together \eqref{c3:control_tb_zeta} and \eqref{c3:bound_zeta}, one gets the Lemma.\qquad$\Box$

\end{proof}

We can now apply Lemma $\ref{c3:control_tb_zeta}$ to get immediately (note that $\vert \nablag\cdot \nabla\zeta\vert_2\sim \vert \nablag\nablag\zeta\vert_2$):
\begin{align}
\vert ((h_b\A)^{-1} (h_b\Bi)\Vu_{k+1},h_b(I-\mu\nablag\nablag\cdot)\nablag\zeta_k)_{H^{N-k-1}}\vert &\leq  C(\hbi,\vert h_b\vert_{H^{N+1}}) (\nonumber\\\sqrt{\mu}\vert \nablag\Vu_{k+1}\vert_{H^{N-k-1}} \vert \nablag\zeta_k \vert_{H^{N-k-1}}+ \sqrt{\mu}\vert \nablag\Vu_{k+1} \vert_{H^{N-k-1}}\sqrt{\mu}\vert  \nablag\cdot \nablag \zeta_k\vert_{H^{N-k-1}})\nonumber\\
 &\leq  C(\hbi,\vert h_b\vert_{H^{N+1}}) \times\nonumber\\(C(K)(t+\epsilon)+C_0)\vert\nablag\zeta_k\vert_{X^{N-k-1}}\label{c3:control_atbm}
 \end{align}
 using the notations of Proposition \eqref{c3:reguA}, and using the induction hypothesis.  Putting \eqref{c3:control_transport_zeta} and \eqref{c3:control_atbm} into \eqref{c3:recover_zeta}, one gets:
 \begin{align*}
 \vert\nablag\zeta_k\vert_{X^{N-k-1}}^2 \leq C(K)\epsilon +(C(K)(t+\epsilon)+C_0)\vert\nablag\zeta_k\vert_{X^{N-k-1}}.
 \end{align*}
 By noticing that, for all $u$ smooth enough:
 $$\vert u\vert_{X^{N-k}}\leq \vert u\vert_{X^{N-k-1}}+\vert \nablag u\vert_{X^{N-k-1}},$$  one finally recovers the "good estimate"  $k$ for $\zeta_k$:
 $$\vert \zeta_k\vert_{H^{N-k}}+\sqrt{\mu}\vert\nablag\zeta_k\vert_{H^{N-k}}\leq C(K)(t+\epsilon)+C_0.$$

Now, we get the "good estimate" for $\Vk$. The equation \eqref{c3:quasilinear_boussi_temps} gives:
$$\nablag\cdot(h_b\Vk) =- h_b (q_{k+1}+\epsilon\Vu\cdot\nablag q_k)$$ and using induction hypothesis to control $q_{k+1}$, one easily gets
$$\vert \nablag\cdot(h_b\Vk)\vert_{H^{N-k-1}} + \sqrt{\mu}\vert \nablag \nablag\cdot(h_b\Vk)\vert_{H^{N-k-1}} \leq C(K)(t+\epsilon)+C_0.$$
If $d=1$, then we controlled a full derivative of $\Vk$, and the induction hypothesis  is true for $k$.\qquad $\Box$

\begin{remark}If $d=2$, of course, it is not sufficient to control only $\nablag\cdot\Vk$ to recover a good control for $\Vk$ in norm $H^{N-k}$. One should look after a good control for $\nablago\cdot\Vk$. This is obtained by taking $\nablago\cdot$ of the second equation of \eqref{c3:quasilinear_boussi_temps}:
\begin{equation}
\dt (\nablago\cdot\Vk) + \dt \mu\nablago\cdot\nablago(\nablago\cdot\Vk)+ (\Vu\cdot\nablag)  \nablago \cdot\Vk = \nablago\cdot (R_2^k)+[\nablago\cdot,\Vu]\Vk - \nablago\cdot\mu\Tb\dt\Vk \label{c3:nablago_equation}
\end{equation}
However, it is difficult to control $\nablago\cdot\mu\Tb\dt\Vk$. Indeed, $\nablago\cdot\Tb$ is no longer symmetric, which means that multiplying the equation \eqref{c3:nablago_equation}  by $\nablago\cdot\Vu$ creates a term of the form $(\nablago\cdot\Tb\partial_t\Vu,\nablago\cdot\Vu)_2$ which is not the time derivative of a positive quantity.
\end{remark}

The key Lemma \ref{c3:keylemma} is this latter result with $k=0$. We now end the proof of Theorem \ref{c3:large_time_theorem}. We set $$ \epsilon_0 = \frac{C_0}{2C(2C_0)},\qquad T_0 =  \frac{C_0}{2C(2C_0)}.$$ Let fix an $\epsilon>0$ such that $\epsilon<\epsilon_0$. There exists $T^\epsilon$ and a unique solution $(\zeta^\epsilon,\Vu^\epsilon)\in C([0;T^\epsilon[;\mathcal{E}^N)$ to the equation \eqref{c3:rescaled_bousisnesq}. We set $$T^\epsilon_* = \underset{t\in[0;T^\epsilon[}{\sup}\lbrace t,\quad (\zeta^\epsilon,\Vu^\epsilon),\text{ exists on } [0;t] \text{ with :} \forall s\leq t, E^N(\zeta^\epsilon,\Vu^\epsilon)(s)\leq 2C_0\rbrace$$ Then, one has $T^\epsilon_* \geq T_0$. Indeed, suppose it is not true.  One has for all $t<T^\epsilon$, using Lemma \ref{c3:keylemma}:
$$E^N(\zeta^\epsilon,\Vu^\epsilon)(t)\leq C(K)(t+\epsilon)+C_0,$$ with $$K=\underset{t\in[0;T^\epsilon_*[}{\sup} E^N(\zeta^\epsilon,\Vu^\epsilon).$$ Notice that $K\leq 2C_0$ by definition of $T^\epsilon_*$. Since $C$ is non decreasing, one has, for all $t\leq T^\epsilon_*$:
\begin{align*}
E^N(\zeta^\epsilon,\Vu^\epsilon)(t) &\leq C(2C_0)(t+\epsilon)+C_0 \\
&<C(2C_0)(T_0 +\epsilon_0)\\
&< 2C_0
\end{align*}
and therefore, by continuity, there exists $\tilde{T}^\epsilon>T^\epsilon_*$ such that $(\zeta^\epsilon,\Vu^\epsilon)$ exists on $[0;\tilde{T}^\epsilon]$ with $E^N(t) \leq 2C_0$ for all $t\leq\tilde{T}^\epsilon$. It is absurd, by definition of $T^\epsilon_*$. Therefore, the solution exists on $[0;T_0]$ which is the result of Theorem \ref{c3:large_time_theorem}.\qquad$\Box$

\end{proof}
\end{proof}

\begin{appendix}

\section{Classical results on Sobolev spaces}
We recall here some classical results on Sobolev spaces. Proofs can be found in \cite{taylor3}. The first result is the Kato-Ponce estimate on commutators:

\begin{proposition}[Kato-Ponce]\label{c3:katoponce}
For all $s\geq 0$ and $f\in H^s\cap W^{1,\infty}$ and $u\in H^{s-1}\cap L^\infty$, one has the following inequality:
$$\vert [\Lambda^s,f]u\vert_2 \leq C (\vert \nabla f\vert_{H^{s-1}}\vert u\vert_\infty + \vert\nablag f\vert_\infty\vert u\vert_{H^{s-1}})$$
where $C$ is a positive constant independent of $f$ and $u$. 
\end{proposition}

The following result stands that one can compose any $H^s\cap L^\infty$ function with a smooth function.

\begin{proposition}[Moser]\label{c3:moser}
Let $F:\R\rightarrow\R$ be a smooth function, null at zero. Then, for all $s\geq 0$, and all $u\in H^s(\R^d)$, $F(u)\in H^s(\R^d)$ and $$\vert F(u)\vert_{H^s} \leq c(\vert u\vert_\infty)\vert u\vert_{H^s}$$ where $c$ is a smooth non decreasing function.
\end{proposition}

\section{Results on the operator $\mathcal{A}$}\label{c3:appendix_operator}
We prove in this section the regularity of the inverse of $h_b\A$ stated by Proposition \ref{c3:reguA}.
\begin{proposition}
Let $N\in\mathbb{N}$ and let $h_b\in H^N(\R^d)$ be such that there exists $h_{\min}>0$ such that $$\forall X\in\R^d,\qquad h_b(X)\geq h_{\min}.$$ We set, for all $f\in L^2(\R^d)$ : $$\vert f\vert_{X^N}^2 = \vert f\vert_{H^N}^2+\mu\vert\nablag\cdot f\vert_{H^N}^2.$$  The operator $h_b\A$ is invertible on $H^N(\R^d)^d$. Moreover, the following estimates stand.\begin{enumerate}[label = ( \arabic*)]\item   For all $f\in H^N(\R^d)^d$,
$$ \vert (h_b\A)^{-1} f\vert_{X^N} \leq C(\hbi,\vert h_b\vert_{H^{N+1}})\vert f\vert_{H^N}.$$
\item For all $g\in H^N(\R^d)$,
$$ \sqrt{\mu}\vert (h_b\A)^{-1} \nablag g\vert_{X^N} \leq C(\hbi,\vert h_b\vert_{H^{N+1}})\vert g\vert_{H^N}.$$\end{enumerate}
\end{proposition}
\begin{proof}
Let $N\in\mathbb{N}$. We define $$X^N = \lbrace V\in L^2(\R^d),\qquad \nablag\cdot V\in L^2(\R^d)\rbrace.$$ Endowed with the scalar product $$(\cdot,\cdot)_{X^N} = (\cdot,\cdot)_{H^N}+\mu(\nablag\cdot,\nablag\cdot)_{H^N},$$ $X^N$ is an Hilbert space with norm $$\vert \cdot\vert_{X^N}^2 = \vert\cdot\vert_{H^N}^2+\mu\vert\nablag\cdot\vert_{H^N}^2.$$ We start to prove that $h_b\A$ is invertible, by using a Lax Milgram's Theorem with the bilinear form $$T:(V_1,V_2)\in X^0\times X^0 \longmapsto (h_b\A V_1, V_2)_2.$$ \begin{enumerate}[label = \roman*)]
\item The bilinear form $T$  is continuous: 
\end{enumerate} Indeed, one has, for all $V_1,V_2\in X^N$ : 
\begin{align*}
(h_b \A V_1,V_2)_2 = (h_b V_1,V_2)_2+\mu(\hbi\nablag\cdot(h_b V_1),\nablag\cdot(h_b V_2))_2
\end{align*}
and therefore one has 
\begin{equation*}
(h_b \A V_1,V_2)_2 \leq C_1(h_{\min},\vert h_b\vert_\Wuf)\vert V_1\vert_{X^0}\vert V_2\vert_{X^0}
\end{equation*}
with $C_1$ a non decreasing function of its arguments. 
\begin{enumerate}[label = \roman*),resume]
\item The bilinear form $T$ is coercive:
\end{enumerate}
Let us write for all $V\in X^0$: \begin{align*}
(h_b\A V,V)_2 = (h_b V,V)_2+\mu(\hbi \nablag\cdot(h_b V),\nablag\cdot(h_b V))_2
\end{align*} which already gives 
\begin{equation}
\vert V\vert_2^2 \leq \frac{1}{h_{\min}}(h_b V,V)_2 .\label{c3:coercive1}
\end{equation}
Moreover, one has
\begin{align*}
(h_b V,V)_2 + (\hbi \nablag\cdot V,\nablag\cdot V)_2 &= (h_b\A V,V)_2-2\mu(\nablag\cdot V,V\cdot \nablag h_b)_2 \\&- \mu(\hbi \nablag (h_b)\cdot V, \nablag (h_b)\cdot V)_2 \\
&\leq (h_b\A V,V)_2+2\mu\vert\nablag\cdot V\vert_2\vert V\vert_2\vert h_b\vert_\Wuf+\frac{1}{h_{\min}}\mu\vert V\vert_2^2\vert h_b\vert_\Wuf 
\end{align*}
and one can conclude using \eqref{c3:coercive1} and Young's inequality that \begin{equation}\vert V\vert_{X^0}^2 \leq C_2(\frac{1}{h_{\min}},\vert h_b\vert_\Wuf)(h_b \A V,V)_2.\label{c3:V_f}  \end{equation}
Using Lax Milgram's Theorem, for all $f\in L^2(\R^d)$, there exists a unique $V_f\in X^0$ be such that $h_b\A V_f = f$. We now prove the first estimate on $V_f$ stated by the Proposition by induction on $N$. Taking $V=V_f$ in \eqref{c3:V_f}, one has this estimate for $N=0$. Let us suppose that the result is true for $N-1$ with $N\geq 1$, and let us prove it for $N$. One has, differentiating $N$ times the relation $h_b \A V_f =f$  (we denote by $\partial^N$ any derivative of order $N$ below):
\begin{align}
h_b \A \partial^N V_f = \partial^N f + [h_b\A, \partial^N] V_f\label{c3:derive_Vf}
\end{align}
and $[h_b\A, \partial^N]V_f$ is sum of terms of the form $$R_1^N =  \mu\partial^{k_1}(h_b) \nablag(\partial^{k_2}(\hbi)\nablag\cdot(\partial^{k_3} (h_b)\partial^{k_4} V_f))$$ and $$R_2^N = \partial^{l_1} (h_b)\partial^{l_2} V_f$$ with $k_1+k_2+k_3+k_4=N$ and $k_4<N$, and with $l_1+l_2=N$ and $l_2<N$. Taking the $L^2$ scalar product of \eqref{c3:derive_Vf} with $\partial^N V_f$, and noticing that \begin{align*}
(R_1^N,\partial^N V_f) &=\mu (\partial^{k_2}(\hbi) \nablag\cdot(\partial^{k_3} (h_b) \partial^{k_4} V_f),\nablag(\cdot \partial^{k_1}(h_b)\partial^N V_f))_2,\end{align*} one gets:
$$(h_b\A\partial^N V_f,\partial^N V_f)_2\leq  C(\hbi,\vert h_b\vert_{H^{N+1}})(\vert\partial^N V_f\vert_2\vert \partial^{l_2} V_f\vert_2+\mu\vert \nablag\cdot \partial^N V_f\vert(\vert\nablag\cdot \partial^{k_4} V_f\vert_2+\vert \partial^{k_4} V_f\vert_2)).$$

Using the induction hypothesis, the terms $\nablag\cdot\partial^{k_4}V_f$ and $\partial^{l_2} V_f$ are already controlled by $$C(\hbi,\vert h_b\vert_{H^{N+1}})\vert f\vert_{H^N},$$ since $k_4<N$ and $l_2<N$. One finally gets with a Young's inequality that \begin{equation}
(h_b \A \partial^N V_f,\partial^N V_f )_2 \leq C(\frac{1}{h_{\min}},\vert h_b\vert_{H^{N+1}})\vert f\vert_{H^N}\label{c3:reguVfN}
\end{equation}
and combining \eqref{c3:reguVfN} with \eqref{c3:V_f}, one gets the estimate of the Theorem by a duality argument. \par\vspace{\baselineskip}

To prove the second point of the Proposition, one  has to notice that for all $f=\sqrt{\mu}\nablag g$ with $g\in H^1(\R^d)$ and all $V\in X^0$ :
$$(f,V)_2 = -(g,\sqrt{\mu}\nablag\cdot(V))_2$$ and one can adapt all the proof of the first point to get the desired result, since $\sqrt{\mu}\nablag\cdot(V)\in L^2(\R^d)$. \qquad $\Box$

\end{proof}

\end{appendix}

The author has been partially funded by the ANR project Dyficolti ANR-13-BS01-0003-01.

\bibliographystyle{plain}

\bibliography{boussinesq_modified}

\end{document}